\documentclass[11pt,reqno]{amsart}

\usepackage{amssymb}
\usepackage{latexsym}
\usepackage{amsmath}

\def\wt{\widetilde}
\def\wh{\widehat}
\def\ov{\overline}
\def \im{{\rm Im}}
 \def\up{\upharpoonright}
\def\cH{\mathcal H}  \def\dcH{\dot{\mathcal H}}

\def\cB{\mathcal B}

\def\cD{\mathcal D}   \def\cG {\mathcal G}
\def\cK{\mathcal K} \def\cL{\mathcal L}
 \def\cN{\mathcal N}  \def\cU{\mathcal U} 
 \def\cT{\mathcal T} \def\cI{\mathcal I}

\def \gH{\mathfrak H}   \def \gN{\mathfrak N}

\def \bC{\mathbb C}    \def\bR{\mathbb R}
\def\bH{\mathbb H} 
\def \l{\lambda}
\def \a{\alpha} \def \b{\beta}  \def \L{\Lambda}  \def \s{\sigma} \def \t{\theta}  \def\g {\gamma}
\def\d {\delta}   \def\Om {\Omega} 
\def \f{\varphi} \def\D {\Delta} \def\Si{\Sigma}
 \def \G{\Gamma} \def \dG{\dot\Gamma}
\def \C{\widetilde {\mathcal C}}
\def \CA{\C(\cH_0,\cH_1)}

\def \cd {\cdot}

\def\BR {\G:\gH\oplus\gH\to \cH_0\oplus\cH_1}

\def\AC {AC(\cI; \bH)}  \def\LI {L_\Delta^2(\cI)}
\def\lI {\cL_\Delta^2(\cI)}
\def\LS {L^2(\s ;\bH_0)} 
\def\LSa {L^2(\s ;\bH_0')}\def\lSa {\cL^2(\s; \bH_0')}

\def\tma{\cT_{\max}}  \def\Tma{T_{\max}} \def\Tmi{T_{\min}}
\def\Sel{\wt{\rm Self}_0 (T)} \def\Selo{{\rm Self}_0 (T) }

\def\bpa{\{\cH_0\oplus\cH_1,\G\}} 

\def \dom {{\rm dom}\,}  \def \ran {{\rm ran}\,}  \def \ker{{\rm ker\,}}
 \def \mul {{\rm mul}\,} 

\def\codim{{\rm codim}\,}
  \def\tm{\times}

\def  \RH {\wt R (\cH_0,\cH_1)}

\def \CR {\bC\setminus\bR}

\newcommand {\lo}[1] {\cL_\D^2[#1,\bH ]}

\def\bta{\{\cH_0\oplus \cH_1,\Gamma _0,\Gamma _1\}}

\newtheorem{theorem}{Theorem}[section]
\newtheorem{proposition}[theorem]{Proposition}
\newtheorem{corollary}[theorem]{Corollary}
\newtheorem{lemma}[theorem]{Lemma}
\newtheorem{assertion}[theorem]{Assertion}
\theoremstyle{definition}

\theoremstyle{definition}
\newtheorem {definition} [theorem]{Definition}
\theoremstyle{remark}
\newtheorem{remark}[theorem]{Remark}
\numberwithin{equation}{section}
\begin{document}
\title[Pseudospectral  functions]
{Spectral and pseudospectral  functions of various dimensions for symmetric systems}
\author {Vadim Mogilevskii}
\email{vadim.mogilevskii@gmail.com}
\begin{abstract}
The main object of the paper is a symmetric system $J y'-B(t)y=\l\D(t) y$  defined  on an interval $\cI=[a,b) $ with the regular endpoint $a$.
Let $\f(\cd,\l)$ be a matrix solution of this system of an arbitrary dimension and let $(Vf)(s)=\int\limits_\cI \f^*(t,s)\D(t)f(t)\,dt$ be the Fourier transform of the function $f(\cd)\in L_\D^2(\cI)$. We define a pseudospectral function of the  system as a matrix-valued distribution function $\s(\cd)$ of the dimension $n_\s$ such that $V$ is a partial isometry from $L_\D^2(\cI)$ to $L^2(\s;\bC^{n_\s})$ with the minimally possible kernel. Moreover, we find the minimally possible value of $n_\s$ and  parameterize all spectral and pseudospectral functions of every possible dimensions $n_\s$  by means of a Nevanlinna boundary parameter. The obtained  results develop the results by  Arov and Dym; A.~Sakhnovich, L.~Sakhnovich and Roitberg; Langer and Textorius.
\end{abstract}
\maketitle
\section{Introduction}
Let  $H$ and $\wh H$ be finite dimensional Hilbert spaces, let $\bH=H\oplus\wh H \oplus H$ and let $[\bH]$ be the set of linear operators in $[\bH]$. Recall that  a non-decreasing left continuous operator (matrix) function $\s(\cd): \bR\to [\bH]$ with $\s(0)=0$ is called a distribution function of the dimension $n_\s:=\dim \bH$.

We consider  symmetric  differential system \cite{Atk,GK}
\begin {equation}\label{1.1}
J y'-B(t)y=\l\D(t) y, \quad t\in\cI, \quad \l\in\bC,
\end{equation}
where  $B(t)=B^*(t)$ and $\D(t)\geq 0$ are  $[\bH]$-valued functions defined on an interval $\cI=[a,b), \;b\leq\infty,$ and integrable on each compact subinterval $[a,\b]\subset\cI$ and
\begin {equation} \label{1.2}
J=\begin{pmatrix} 0 & 0&-I_H \cr 0& i I_{\wh H}&0\cr I_H&
0&0\end{pmatrix}:H\oplus\wh H\oplus H \to H\oplus\wh H\oplus H.
\end{equation}
System \eqref{1.1} is called a Hamiltonian system if $\wh H=\{0\}$ and hence $\bH=H\oplus H$,
\begin {equation*}
J=\begin{pmatrix} 0 & -I_H \cr I_H&
0\end{pmatrix}:H\oplus H \to H\oplus H.
\end{equation*}
The system  is called regular if $b<\infty$ and  $\int_\cI ||B(t)||\, dt<\infty, \; \int_\cI ||\D(t)||\, dt<\infty$ .

As is known a spectral function is a basic concept in the theory of eigenfunction expansions of differential operators (see e.g. \cite{Nai,DunSch} and references therein). In the case of a symmetric system definition of the spectral function requires a certain modification. Namely, let  $\gH=L_\D^2(\cI)$ be the Hilbert space of functions $f:\cI\to \bH$ satisfying

\centerline{ $ \int\limits_\cI (\D(t)f(t),f(t))\,dt <\infty.$ }

\noindent Assume that system \eqref{1.1} is Hamiltonian and  $\f(\cd,\l)$ is an $[H,H \oplus H]$-valued operator solution of  \eqref{1.1} such that $\f(0,\l)=(0,I_H)^\top$. An $[H]$-valued distribution function $\s(\cd)$ is called a spectral function of the system  if the (generalized) Fourier transform $V_\s:\gH\to L^2(\s; H)$ defined by
\begin {equation} \label{1.4}
(V_\s f)(s)=\wh f_0(s):=\int_\cI \f^*(t,s)\D(t)f(t)\,dt, \quad f(\cd)\in\gH
\end{equation}
is an isometry. If $\s(\cd)$ is a spectral function, then the inverse Fourier transform is defined for each $f\in\gH$ by
\begin {equation} \label{1.5}
f(t)=\int_{\cI}\f(t,s)\, d\s(s)\wh f_0(s)
\end{equation}
(the integrals in \eqref{1.4} and \eqref{1.5} converge in the norm of $L^2(\s; H)$ and $\gH$ respectively). If the operator $\D(t)$ is invertible a.e. on $\cI$, then spectral functions  exist. Otherwise the Fourier transform may have a nontrivial kernel $\ker V_\s$  and hence the set of spectral functions may be empty. The natural generalization of a spectral function to this case is an $[H]$-valued distribution function $\s(\cd)$ such that the Fourier transform $V_\s$ of the form \eqref{1.4} is a partial isometry. If $\s(\cd)$ is such a function, then the inverse Fourier transform \eqref{1.5} is valid for each $f\in\gH\ominus \ker V_\s$. Therefore an interesting problem is a characterization of $[H]$-valued distribution functions $\s(\cd)$ such that the  Fourier transform  $V_\s$ is a partial isometry with the minimally possible kernel $\ker V_\s$. This problem was solved in \cite{AD,Sah90,Sah13} for regular systems and in \cite{Mog15.2} for general systems. The results of \cite{Mog15.2} was obtained in the framework of the extension theory of symmetric linear relations.  As is known \cite{Orc,Kac03,LesMal03,Mog12} system \eqref{1.1} generates the minimal (symmetric) linear relation $\Tmi$ and the maximal relation $\Tma(=\Tmi^*)$ in $\gH$. Let $T\supset \Tmi$ be a symmetric relation in $\gH$ given by
\begin {equation*}
T=\{\{y,f\}\in\Tma:(I_H,0)y(a)=0\;\;{\rm and}\;\;\lim_{t\to b}(J y(t),z(t))=0, \;\; z\in\dom\Tma\}
\end{equation*}
and let $\mul T$ be the multivalued part of $T$. It was shown in \cite{Mog15.2} that for each $[H]$-valued distribution function $\s(\cd)$ such that $V_\s$ is a partial isometry the inclusion $\mul T \subset \ker V_\s$ is valid. This fact makes natural the following definition.
\begin{definition}\label{def1.1}$\:$ \cite{Mog15.2}
An $[H]$-valued distribution function $\s(\cd)$  is called a pseudospectral function of the system \eqref{1.1} (with respect to $\f(\cd,\l)$ ) if the Fourier transform $V_\s$ is a partial isometry with the minimally possible kernel  $\ker V_\s=\mul T$.
\end{definition}
If the Hamiltonian system is regular,  then $\ker V_\s=\{f\in\gH: \wh f_0(s)=0,\; s\in\bR\}$ and therefore Definition \ref{def1.1} turns into the definition of the pseudospectral function from the monographes \cite{AD,Sah13}. In these monographes all $[H]$-valued pseudospectral functions of the regular system are parameterized  in the form of a linear fractional transform of a Nevanlinna parameter. Similar result for singular systems was obtained in our paper \cite{Mog15.2}. Observe also that an existence of $[H]$-valued pseudospectral functions of the singular Hamiltonian system in the case $\dim H=1$ was proved in \cite{Kac03}.

Assume now that system \eqref{1.1} is not necessarily Hamiltonian. Let  $Y(\cd,\l)$ be the $[\bH]$-valued operator solution of  \eqref{1.1} with $Y(a,\l)=I_{\bH}$ and let  $\Si(\cd)$ be an $[\bH]$-valued distribution function such that  the Fourier transform $V_\Si:\gH\to L^2(\Si;\bH)$ defined by
\begin {equation} \label{1.6}
(V_\Si f)(s)=\wh f(s):=\int_\cI Y^*(t,s)\D(t)f(t)\,dt, \quad f(\cd)\in\gH
\end{equation}
is a partial  isometry. Moreover, let $\mul\Tmi$ be the multivalued part of $\Tmi$. Then according to \cite{Mog15} $\mul\Tmi\subset \ker V_\Si$ and the same arguments as for transform \eqref{1.4} make natural the following definition.
\begin{definition}\label{def1.2}$\,$
\cite{Mog15} An $[\bH]$-valued distribution function $\Si(\cd)$  is called a pseudospectral function of the system \eqref{1.1} (with respect to $Y(\cd,\l)$) if the Fourier transform $V_\Si$ is a partial isometry with the minimally possible kernel  $\ker V_\Si=\mul\Tmi$.
\end{definition}
Existence of pseudospectral functions $\Si(\cd)$ follows from the results of \cite{DLS88,DLS93,LanTex84,LanTex85}. In \cite{LanTex84,LanTex85} a parametrization of  all pseudospectral functions $\Si(\cd)$  of the regular system \eqref{1.1} is given. This parametrization is closed to that of the $[H]$-valued pseudospectral functions $\s(\cd)$ in \cite{AD,Sah13}. Similar result for singular systems is obtained in \cite{Mog15}.

In the present paper we continue our investigations of pseudospectral and spectral functions of symmetric systems contained in \cite{AlbMalMog13,Mog15,Mog15.2}.

According to Definitions \ref{def1.1} and \ref{def1.2} the dimensions of pseudospectral functions $\Si(\cd)$ and $\s(\cd)$ are $n_\Si=\dim\bH$ and $n_\s=\dim H (< n_\Si)$. In this connection the following problems seems to be interesting:
\vskip 1mm
\noindent $\bullet\;\;$  To define naturally a spectral and pseudospectral function of an arbitrary dimension for the system \eqref{1.1} and describe all such functions by analogy with \cite{Mog15,Mog15.2}.
\vskip 1mm
\noindent $\bullet\;\;$ To characterize spectral functions of the minimally possible dimension
\vskip 1mm
The paper is devoted to the solution of these problems.

Let $\bH_0$ and $\t$ be subspaces in $\bH$, $K_\t\in [\bH_0,\bH]$ be an operator isomorphically mapping $\bH_0$ onto $\t$ and $\f(\cd,\l)$ be the $[\bH_0,\bH]$-valued operator solution of \eqref{1.1} with $\f(a,\l)=K_\t$. Moreover, let $\s(\cd)$ be an $[\bH_0]$-valued distribution function such that the Fourier transform $V_\s:\bH\to L^2(\s;\bH_0)$ defined by \eqref{1.4} is a partial isometry. It turns out that $\mul T \subset \ker V_\s$, where $\mul T$ is the multivalued part of a symmetric relation $T\supset \Tmi$ in $\gH$ given by
\begin {equation*}
T=\{\{y,f\}\in\Tma: y(a)\in \t\;\;{\rm and}\;\;\lim_{t\to b}(J y(t),z(t))=0, \;\; z\in\dom\Tma\}.
\end{equation*}
This statement makes natural the following most general definition of pseudospectral and  spectral functions.
\begin{definition}\label{def1.3}
An $[\bH_0]$-valued distribution function $\s(\cd)$  is called a pseudospectral function of the system \eqref{1.1} (with respect to the operator  $K_\t$) if the  Fourier transform $V_\s$ is a partial isometry with the minimally possible kernel  $\ker V_\s=\mul T$.

A pseudospectral function $\s(\cd)$ with $\ker V_\s=\{0\}$ is called a spectral function.
\end{definition}
It turns out that actually a pseudospectral function with respect to  the operator $K_\t $ is uniquely characterized by the subspace $\t\subset\bH$.

We parametrize all pseudospectral (spectral) functions for a given $\t$ and find a lower bound of the dimension of all spectral functions $\s(\cd)$ corresponding to various $\t$. More precisely the following three theorems are the  main results of the paper.
\begin{theorem}\label{th1.4}
Assume that system \eqref{1.1} is definite (see Definition \ref{def3.11.1}) and deficiency indices $n_\pm(\Tmi)$ of $\Tmi$ satisfy $n_-(\Tmi)\leq n_+(\Tmi)$. Moreover, let $\t$ be a subspace in $\bH$ and let $\t^\tm:=\bH\ominus J\t$. Then a pseudospectral function $\s(\cd)$ (with respect to $K_\t$) exists if and only if $\t^\tm \subset \t$.
\end{theorem}
\begin{theorem}\label{th1.5}
Assume that $\t$ is a subspace in $\bH$ such that $\t^\tm \subset \t$ and there exists only a trivial solution $y=0$ of the system \eqref{1.1} such that $\D(t)y(t)=0$ (a.e. on $\cI$) and  $y(a)\in\t$ (the last condition is fulfilled for definite systems). Moreover, let for simplicity $n_+(\Tmi)=n_-(\Tmi)$. Then:

{\rm (1)} There exist   auxiliary finite-dimensional Hilbert spaces  $\bH_0\subset\bH$ and  $\dot\cH$,  an operator $U=U_\t\in [\bH_0,\bH]$ isomorphically mapping $\bH_0$ onto $\t$, Nevanlinna operator functions $m_0(\l)(\in [\bH_0])$, $\dot M(\l)(\in [\dot\cH])$ and an operator function $ S(\l)(\in [\dot\cH,\bH_0])$  such that the equalities
\begin {gather}
m_\tau(\l)=m_0(\l)+S(\l)(C_0(\l)-C_1(\l)\dot M(\l))^{-1}C_1(\l)
S^*(\ov\l), \quad \l\in\CR\label{1.8}\\
\s_\tau(s)=\lim\limits_{\d\to+0}\lim\limits_{\varepsilon\to +0} \frac 1 \pi\int_{-\d}^{s-\d}\im \,m_\tau(u+i\varepsilon)\, du\label{1.9}
\end{gather}
establish a bijective correspondence $\s(s)=\s_\tau(s)$ between all  Nevanlinna operator pairs $\tau=\{C_0(\l),C_1(\l)\}$, $ C_j(\l)\in [\dot\cH],\; j\in\{0,1\},$ satisfying the admissibility conditions
\begin {gather}
\lim_{y\to \infty} \tfrac 1 {i y} (C_0(i y)-C_1(i y )\dot M(i y))^{-1}C_1(i y) =0\label{1.10}\\
\lim_{y\to \infty} \tfrac 1 {i y} \dot M(i y)( C_0(i y)-C_1(i y )\dot M(i y) )^{-1} C_0(i y)=0\label{1.11}
\end{gather}
and all pseudospectral functions  $\s (\cd)$ of the system (with respect to $U$). Moreover, each pair $\tau$ is admissible (and hence the conditions \eqref{1.10} and \eqref{1.11} may be omitted)  if and only if $\lim\limits_{y\to \infty} \tfrac 1 {iy}\dot M(iy)=0$ and $\lim\limits_{y\to\infty} y \cd \im (\dot M(iy)h,h)=
+\infty,\quad  0\neq h\in\dot\cH$.

{\rm (2)} The set of spectral functions (with respect to $U$) is not empty if and only if $\mul T=\{0\}$. If this condition id fulfilled, then the sets of spectral and spectral function (with respect to $U$) coincide and hence statement {\rm (1)} holds for spectral functions.
\end{theorem}
\begin{theorem}\label{th1.6}
Let system \eqref{1.1} be definite  and let  $n_-(\Tmi)\leq n_+(\Tmi)$.
Then the set of spectral functions of the system is not empty if and only if $\mul \Tmi=\{0\}$. If this condition is fulfilled, then the dimension $n_\s$ of each spectral function $\s(\cd)$ satisfies $\dim (H\oplus\wh H)\leq n_\s\leq \dim \bH$ and there exists a subspace $\t\subset \bH$ and a spectral function $\s(\cd)$ (with respect to $K_\t$) such that the dimension $n_\s$ of $\s(\cd)$ has the minimally possible value $ n_\s=\dim (H\oplus\wh H)$.
\end{theorem}
Note that the coefficients $m_0(\l), \; S(\l)$ and $\dot M (\l)$ in \eqref{1.8} are defined in terms of the  boundary values of respective operator solutions of  \eqref{1.1} at the endpoints $a$ and $b$. Observe also that $m_\tau(\l)$ in \eqref{1.8} is an $[\bH_0]$-valued Nevanlinna function (the $m$-function of the system) and  \eqref{1.9} is the Stieltjes formula for $m_\tau(\cd)$. If the system is Hamiltonian, $\t$ is a self-adjoint linear relation in $H\oplus H$ and $\tau=\tau^*$, then $m_\tau(\l)$ is the Titchmarsh - Weyl function of the system corresponding to self-adjoint separated boundary conditions \cite{HinSch93}. In the case of a non-Hamiltonian system such conditions do not exist \cite{Mog12} and $m_\tau(\l)$ corresponds to special mixed boundary conditions (see Definition \ref{def4.13}).

For pseudospectral functions $\s(\cd)$ of the minimal dimension $ n_\s=\dim (H\oplus\wh H)$ formulas similar to \eqref{1.8} and \eqref{1.9} were obtained in \cite{AlbMalMog13}. These formulas are proved in \cite{AlbMalMog13} only for a parameter $\tau$ of a special form; therefore not all pseudospectral functions  $\s(\cd)$  are parametrize in this paper.

As is known \cite{KogRof75,Mog15.3} the set of spectral functions of a symmetric differential operator $l[y]$ of an order $m$ coincides with the set of spectral functions of a special definite symmetric system corresponding to  $l[y]$. Moreover, this system is Hamiltonian if and only if $m$ is even. According to the classical monograph by N.~Dunford and J.T.~Schwartz \cite[ch. 13.21]{DunSch} an important problem of the spectral theory of differential operators  is a characterization of their spectral functions $\s_{\min}(\cd)$ with the minimally possible dimension $n_{\min}$. It follows from Theorem \ref{th1.6} that $n_{\min}=k+1$ in the case $m=2k+1$ and $n_{\min}=k$ in the case $m=2k$. Moreover, by using Theorem \ref{th1.5} one may obtain a parametrization of $\s_{\min}(\cd)$. In more details this results will be specified elsewhere.

For a differential operator $l[y]$ of an even order $m$ formulas similar to \eqref{1.8} and \eqref{1.9} were proved in our paper \cite{Mog12.1}. These formulas enable one to calculate spectral functions $\s(\cd)$ of an arbitrary dimension $n_\s$ ($\tfrac m 2 \leq n_\s\leq m$) corresponding to a special parameter $\tau$; hence they do not parametrize all spectral functions of $l[y]$.

In conclusion note that our approach is based on the theory of boundary triplets (boundary pairs) for symmetric linear relations and their Weyl function (see \cite{Bru77,DM06,DM91,GorGor,Mal92,Mog12} and references therein).

\section{Preliminaries}
\subsection{Notations}
The following notations will be used throughout the paper: $\gH$, $\cH$ denote Hilbert spaces; $[\cH_1,\cH_2]$  is the set of all bounded linear operators defined on $\cH_1$ with values in  $\cH_2$; $[\cH]:=[\cH,\cH]$; $\bC_+\,(\bC_-)$ is the upper (lower) half-plane  of the complex plane. If $\cH$ is a subspace in $\wt\cH$, then $P_{\cH}(\in [\wt\cH])$ denote the orthoprojection in $\wt\cH$ onto $\cH$ and $P_{\wt\cH,\cH} (\in [\wt\cH,\cH])$  denote the same orthoprojection considered as an operator from $\wt\cH$ to $\cH$.

Recall that a linear relation $T: \cH_0\to \cH_1$ from a Hilbert space  $\cH_0$ to a Hilbert space $\cH_1$ is a linear manifold  in the Hilbert $\cH_0\oplus\cH_1$. If $\cH_0=\cH_1=:\cH$ one speaks of a linear relation
$T$ in $\cH$. The set of all closed linear relations from $\cH_0$ to $\cH_1$ (in $\cH$) will be denoted by $\C (\cH_0,\cH_1)$ ($\C(\cH)$). A closed linear operator $T$ from $\cH_0$ to $\cH_1$  is
identified  with its graph $\text {gr}\, T\in\CA$.

For a linear relation $T\in\C (\cH_0,\cH_1)$  we denote by $\dom T,\,\ran T,
\,\ker T$ and $\mul T$  the domain, range, kernel and the multivalued part of
$T$ respectively. Recall that $\mul T$ ia a subspace in $\cH_1$ defined by
\begin{gather}\label{2.1}
\mul T:=\{h_1\in \cH_1:\{0,h_1\}\in T\}.
\end{gather}
Clearly, $T\in \C (\cH_0,\cH_1)$ is an operator if and only if $\mul T=\{0\}$. For $T\in \C (\cH_0,\cH_1)$ we will denote by $T^{-1}(\in\C (\cH_1,\cH_0))$ and $T^*(\in\C (\cH_1,\cH_0))$ the inverse and adjoint linear relations of $T$ respectively.

Recall  that an operator function $\Phi
(\cd):\bC_+\to [\cH]$ is called a Nevanlinna function (and referred to the class $R[\cH]$) if it is
holomorphic and  $ \im \Phi (\l)\geq 0, \;\l\in\bC_+$.
\subsection{Symmetric relations and generalized resolvents}
As is known a linear relation $A\in\C (\gH)$ is called symmetric (self-adjoint) if $A\subset A^*$ (resp. $A=A^*$). For each symmetric relation $A\in\C (\gH)$ the following decompositions hold
\begin{equation}\label{2.2}
\gH=\gH'\oplus \mul A, \qquad A={\rm gr}\, A'\oplus \wh {\mul} A,
\end{equation}
where $\wh {\mul} A=\{0\}\oplus \mul A$ and $A'$ is a closed symmetric not necessarily densely defined operator in $\gH'$ (the operator part of $A$). Moreover, $A=A^*$ if and only if $A'=(A')^*$.

Let $A=A^*\in \C (\gH)$, let $\cB$ be the Borel $\sigma$-algebra of $\bR$ and let $E_0(\cd):\cB\to [\gH_0]$ be the orthogonal spectral measure of $A_0$. Then the spectral measure $E_A(\cd):\cB\to [\gH]$ of $A$ is defined as $E_A(B)=E_0(B)P_{\gH'}, \; B\in\cB$.
\begin{definition}\label{def2.0}
Let $\wt A=\wt A^*\in \C(\wt\gH)$ and let $\gH$ be a subspace in $\wt\gH$. The relation $\wt A$ is called $\gH$-minimal if there is no a nontrivial subspace $\gH'\subset \wt \gH\ominus\gH$ such that $E_{\wt A}(\d)\gH'\subset \gH'$ for each bounded interval $\d=[\a,\b)\subset \bR$.
\end{definition}
\begin{definition}\label{def2.1}
The relations $T_j\in \C (\gH_j), \; j\in\{1,2\},$ are said to be unitarily equivalent (by means of a unitary operator $U\in [\gH_1,\gH_2]$) if $T_2=\wt U
T_1$ with $\wt U=U\oplus U \in [\gH_1^2, \gH_2^2]$.
\end{definition}
Let $A\in\C (\gH)$ be a symmetric relation. Recall the following definitions and results.
\begin{definition}\label{def2.2}
A relation $\wt A=\wt A^*$ in a Hilbert space $\wt \gH \supset \gH$ satisfying $A\subset \wt A$ is called an exit space self-adjoint extension of $A$. Moreover, such an extension $\wt A$ is called minimal if it is $\gH$-minimal.
\end{definition}
In what follows we denote by  $\wt {\rm Self} (A)$ the set of all minimal exit space self-adjoint extensions of $A$. Moreover, we denote by ${\rm Self} (A)$ the set of all extensions $\wt A=\wt A^*\in \C (\gH)$ of $A$ (such an extension is called canonical). As is known, for each $A$ one has $\wt {\rm Self} (A)\neq \emptyset $. Moreover,
${\rm Self} (A)\neq  \emptyset$ if and only if $A$ has equal deficiency indices, in which case ${\rm Self} (A)\subset \wt{\rm Self} (A)$.
\begin{definition}\label{def2.4}
Exit space extensions $\wt A_j=\wt A_j^*\in \C (\wt \gH_j),\; j\in\{1,2\},$ of $A$ are called equivalent (with respect to $\gH$) if there  exists  a unitary operator $V\in
[\wt\gH_1\ominus \gH, \wt\gH_2\ominus \gH]$ such that  $\wt A_1$ and
$\wt A_2$ are unitarily equivalent by means of $U=I_{\gH}\oplus V$.
\end{definition}
\begin{definition}\label{def2.5}
The operator functions $R(\cd):\CR\to [\gH]$ and $F(\cd):\bR\to [\gH]$
are called a generalized resolvent and a spectral function of  $A$
respectively if there exists an exit space self-adjoint extension $\wt A$ of $A$ (in a certain Hilbert space $\wt \gH\supset \gH$)  such that
\begin {gather}
R(\l) =P_\gH (\wt A- \l)^{-1}\up \gH, \quad \l \in \CR\label{2.4}\\
F(t)=P_{\wt\gH,\gH}E_{\wt A}((-\infty,t))\up\gH, \quad  t\in\bR.\label{2.5}
\end{gather}
\end{definition}
\begin{proposition}\label{pr2.6}
Each generalized resolvent $R(\l)$ of $A$ is generated by some (minimal) extension $\wt A\in \wt{\rm Self} (A)$. Moreover,  the extensions $\wt A_1,\, \wt A_2\in \wt {\rm Self} (A)$ inducing the same generalized resolvent $R(\cd)$ are equivalent.
\end{proposition}
In the sequel we suppose   that  a generalized resolvent $R(\cd)$ and a spectral function $F(\cd)$ are generated by an extension $\wt A\in \wt {\rm Self} (A)$. Moreover, we identify equivalent extensions. Then by Proposition \ref{pr2.6} the equality \eqref{2.4} gives a bijective correspondence between generalized resolvents $R(\l)$ and extensions
$\wt A\in \wt {\rm Self} (A)$, so that each $\wt A\in \wt {\rm Self} (A)$  is uniquely  defined by the corresponding generalized resolvent  \eqref{2.4} (spectral function \eqref{2.5}).
\begin{definition}\label{def2.7}
An extension $\wt A\in \wt {\rm Self} (A) $ ($\wt A\in  {\rm Self} (A) $) belongs to the class $ \wt {\rm Self}_0 (A) $ (resp. $   {\rm Self}_0 (A) $) if $\mul\wt A=\mul A$.
\end{definition}
It follows from \eqref{2.2} that the operator $A'$ is densely defined if and only if $\mul A=\mul A^*$. This yields the equivalence
\begin {gather}\label{2.8}
\wt{\rm Self} (A)=\wt{\rm Self}_0 (A) \iff \mul A=\mul A^*
\end{gather}
\subsection{The classes $\wt R(\cH_0,\cH_1)$ and $\wt R(\cH)$}
In the following $\cH_0$ is  a Hilbert space,  $\cH_1$ is a subspace in $\cH_0$, $\cH_2:=\cH_0\ominus\cH_1$, $P_1:=P_{\cH_0,\cH_1}$ and $P_2=P_{\cH_2}$.
\begin{definition}\label{def2.9.1}
$\,$\cite{Mog13.2} A function $\tau(\cd):\bC_+\to\CA$ is referred to the class $\RH$ if:

(i) $2\im (h_1,h_0)-||P_2 h_0||^2\geq 0,\; \{h_0,h_1\}\in\tau(\l),\; \l\in\bC_+$;

(ii) $(\tau(\l)+i P_1)^{-1}\in [\cH_1,\cH_0], \;\l\in\bC_+,$ and the operator-function $(\tau(\l)+i P_1)^{-1}$ is holomorphic on $\bC_+$.
\end{definition}
According to \cite{Mog13.2} the equality
\begin {equation}\label{2.14}
\tau (\l)=\{C_0(\l), C_1(\l)\}:=\{\{h_0,h_1\}\in
\cH_0\oplus\cH_1: C_0(\l)h_0+C_1(\l)h_1=0\}, \;\; \l\in \bC_+
\end{equation}
establishes a bijective correspondence between all functions $\tau(\cd)\in\RH$ and all pairs of holomorphic operator-functions  $C_j(\cd):\bC_+\to [\cH_j,\cH_0], \; j\in\{0,1\},$ satisfying
\begin {equation}\label{2.14.1}
2\,\im(C_{1}(\l)P_1 C_{0}^*(\l))+ C_{0}(\l)P_2 C_{0}^*(\l)\geq
0,\quad
(C_{0}(\l)-iC_{1}(\l)P_1)^{-1}\in [\cH_0], \quad \l\in\bC_+.
\end{equation}
This fact enables one to identify a function $\tau(\cd)\in \wt R (\cH_0,\cH_1)$ and the corresponding pair of operator-functions $C_j(\cd)$ (more precisely the equivalence class of such pairs \cite{Mog13.2}).

If $\cH_1=\cH_0=:\cH$, then the class $\wt R (\cH,\cH)$ coincides with the well-known class $\wt R(\cH)$  of Nevanlinna $\C (\cH)$-valued functions (Nevanlinna operator pairs) $\tau(\l)=\{C_0(\l),C_1(\l)\}, \quad \l\in\bC_+$. In this case the class $\wt R^0(\cH)$ is defined as the set of all $\tau(\cd)\in \wt R(\cH)$ such that
\begin {equation}\label{2.14.2}
\tau(\l)\equiv \t= \{C_0,C_1\},\quad \l\in\bC_+,
\end{equation}
with $\t=\t^*\in \C(\cH)$ and  $C_j\in [\cH]$ satisfying $\im (C_1C_0^*)=0 $ and  $ (C_0\pm i C_1)^{-1}\in [\cH]$.
\subsection{Boundary triplets and boundary pairs}
Here we recall some facts about  boundary triplets  and boundary pairs  following \cite{Bru77,DM06,DM91,GorGor,Mal92,Mog06.2,Mog12}.

Assume that $A$ is a closed  symmetric linear relation in the Hilbert space $\gH$,  $\gN_\l(A)=\ker (A^*-\l)\; (\l\in\bC)$ is a defect subspace of $A$,  $\wh\gN_\l(A)=\{\{f,\l f\}:\, f\in \gN_\l(A)\}$ and  $n_\pm (A):=\dim \gN_\l(A)\leq\infty, \; \l\in\bC_\pm,$ are  deficiency indices of $A$.
\begin{definition}\label{def2.10}
 A collection $\Pi=\bta$, where
$\G_j: A^*\to \cH_j, \; j\in\{0,1\},$ are linear mappings, is called a boundary
triplet for $A^*$, if the mapping $\G :\wh f\to \{\G_0 \wh f, \G_1 \wh f\}, \wh
f\in A^*,$ from $A^*$ into $\cH_0\oplus\cH_1$ is surjective and the following Green's identity
\begin {equation}\label{2.14.3}
(f',g)-(f,g')=(\G_1  \wh f,\G_0 \wh g)_{\cH_0}- (\G_0 \wh f,\G_1 \wh
g)_{\cH_0}+i (P_2\G_0 \wh f,P_2\G_0 \wh g)_{\cH_2}
\end{equation}
 holds for all $\wh
f=\{f,f'\}, \; \wh g=\{g,g'\}\in A^*$.
\end{definition}
A boundary triplet $\Pi=\bta$ for $A^*$ exists if and only if $n_-(A)\leq n_+(A)$, in which case $\dim \cH_1=n_-(A)$ and $ \dim \cH_0 =n_+(A)$.
\begin{proposition}\label{pr2.11}
Let  $\Pi=\bta$ be a boundary triplet for $A^*$ and let $\pi_1$ be the orthoprojection in $\gH\oplus\gH$ onto $\gH\oplus \{0\}$.  Then  the equalities
\begin{gather}
\g_{+} (\l)=\pi_1(\G_0\up\wh \gN_\l (A))^{-1}, \;\;\l\in\Bbb C_+;\quad \g_{-}
(\l)=\pi_1(P_1\G_0\up\wh\gN_\l (A))^{-1}, \;\; \l\in\Bbb C_-
\label{2.15}\\
M_+(\l)h_0=\G_1\{\g_+(\l)h_0,\l \g_+(\l)h_0\},\quad h_0\in\cH_0,\;\;\l\in\bC_+
\label{2.16}
\end{gather}
correctly define holomorphic operator functions $\g_{+}(\cdot):\Bbb C_+\to[\cH_0,\gH], \; \; \g_{-}(\cdot):\Bbb C_-\to[\cH_1,\gH]$ ($\g$-fields of $\Pi$) and $M_{+}(\cdot):\bC_+\to [\cH_0,\cH_1]$ (the Weyl function of $\Pi$).
\end{proposition}
$\g$-field $\g_+(\cd)$ ($\g_-(\cd)$) can be also defined as a unique $[\cH_0,\gH]$-valued (resp. $[\cH_1,\gH]$-valued) operator function such that $\g_+(\l)\cH_0\subset \gN_\l(A)$ (resp. $\g_-(\l)\cH_1\subset \gN_\l(A)$ ) and
\begin{gather}
\G_0 \{\g_+(\l)h_0,\l \g_+(\l)h_0\}=h_0,\;\; h_0\in\cH_0,\;\l\in\bC_+\label{2.16.1}\\
P_1\G_0 \{\g_-(\l)h_1,\l \g_-(\l)h_1\}=h_1,\;\; h_1\in\cH_1,\;\l\in\bC_-.\label{2.16.2}
\end{gather}
A boundary pair for $A^*$ is a generalization of a boundary triplet. Namely, a pair $\bpa$ with  a linear relation $\BR$ is  called a boundary pair for $A^*$ if $\ov{\dom \G} =A^*$, the identity
\begin {equation}\label{2.21}
(f',g)_\gH -  (f,g')_\gH= (h_1,x_0)_{\cH_0} - (h_0,x_1)_{\cH_0}+i(P_2 h_0,P_2 x_0)_{\cH_0}
\end{equation}
holds for every $\{f \oplus f', h_0\oplus h_1\}, \; \{g \oplus g', x_0\oplus x_1\}\in \G $ and a certain maximality condition is satisfied  \cite{DM06,Mog12}. The following proposition is immediate from \cite[Section 3]{Mog12}.
\begin{proposition}\label{pr3.15}
Let  $\bpa$ be a boundary pair for $A^*$ with $\dim\cH_0<\infty$ and  let  $\G_0:\gH\oplus\gH\to \cH_0$ be the linear relations, given by
$\G_0=P_{\cH_0\oplus \{0\}}\G$. Moreover, let
\begin {gather}\label{2.21.1}
\cK_\G=\mul (\mul\G)=\{ h_1\in\cH_1: \{0\oplus 0,0\oplus h_1\}\in\G\}, \qquad \cK_\G\subset \cH_1.
\end{gather}
Then: (1) $\dom\G=A^*$;

(2) If $\cK_\G=\{0\}$,then $\ran \G_0\up\wh\gN_\l(A)=\cH_0, \; \l\in\bC_+;$ $\;\;\ran P_1\G_0\up\wh\gN_\l(A)=\cH_1, \; \l\in\bC_-,$
and the equality
\begin {multline}\label{2.22}
{\rm gr}\, M_+(\l)
=\{ h_0\oplus h_1: \{f\oplus \l f,h_0\oplus h_1\}\in\G\;\; \text{with  some} \;\; f\in\gN_\l(A)\}, \quad \l\in\bC_+
\end{multline}
defines the operator function $M_+(\cd):\bC_+\to [\cH_0,\cH_1] $ (the Weyl function of the  pair $\bpa$). Moreover,
\begin {multline}\label{2.23}
{\rm gr}\, M_+^*(\ov\l)=\\
=\{ \{P_1 h_0\oplus ( h_1+ i P_2 h_0\}: \{f\oplus \l f,h_0\oplus h_1\}\in\G\;\; \text{with  some} \;\; f\in\gN_\l(A)\}, \quad \l\in\bC_-.
\end{multline}
\end{proposition}
\section{Pseudospectral and spectral functions of symmetric systems}
\subsection{Notations}
For an interval $\cI=[ a,b\rangle\subset \bR$ and a  finite-dimensional Hilbert space $\bH$ we denote by $AC(\cI;\bH)$ the set of all functions $f(\cd):\cI\to \bH$, which are absolutely
continuous on each segment $[\a,\b]\subset \cI$.

Assume that $\D(\cd):\cI\to [\bH]$ is a locally integrable  function  such that $\D(t)\geq 0$ a.e. on $\cI$. Denote  by $\lI$  the semi-Hilbert  space of  Borel measurable functions $f(\cd): \cI\to \bH$ satisfying $\int_{\cI}(\D
(t)f(t),f(t))_\bH \,dt<\infty$ (see e.g. \cite[Chapter 13.5]{DunSch}).  The
semi-definite inner product in $\lI$ will be denoted $(\cd,\cd)_\D$. Moreover, let $\LI$ be the Hilbert space of  equivalence classes in $\lI$ with respect
to the semi-norm in $\lI$,  $\pi_\D$ be  the quotient map from $\lI$ onto $\LI$ and   $\wt \pi_\D\{f,g\}:=\{\pi_\D f,
\pi_\D g\}, \;\{f,g\} \in (\lI)^2$.

For a given finite-dimensional Hilbert space $\cK$ we denote by $\cL_\D^2[\cK,\bH]$ the set
of all Borel measurable  operator-functions $F(\cd): \cI\to [\cK,\bH]$ such
that $F(t)h\in \lI, \; h\in\cK$.

In the following for a distribution function $\s(\cd): \bR\to [\bH]$ we denote by $\cL^2(\s;\bH)$ the semi-Hilbert space  of  Borel-measurable functions $g(\cd):\bR\to \bH$  such that $\int_\bR (d\s(s)g(s),g(s)) (s)<\infty$ and by $L^2(\s;\bH)$ the  a Hilbert space of all equivalence classes in $\cL^2(\s;\bH)$ with respect to the seminorm $||\cd||_{\cL^2(\s;\bH)}$ (see e.g. \cite[Chapter 13.5]{DunSch}). Moreover, we denote by $\pi_\s$ the quotient map from $\cL^2(\s;\bH)$ onto $L^2(\s;\bH)$.
\subsection{Symmetric  systems}
Let $H$ and $\wh H$ be  finite dimensional Hilbert
spaces and  let
\begin{gather}
 \bH:=H\oplus\wh H \oplus H\label{3.1}\\
 \nu=\dim H, \qquad \wh \nu= \dim \wh H, \qquad  n = \dim \bH=2\nu+\wh\nu.\label{3.2}
\end{gather}
A first order
symmetric system of differential equations on an interval
$\cI=[a,b\rangle, -\infty<a <b\leq\infty,$ (with the regular
endpoint $a$) is
 of the form
\begin {equation}\label{3.3}
J y'-B(t)y=\l\D(t)y, \quad t\in\cI,\quad  \l\in\bC,
\end{equation}
where $J$ is the operator \eqref{1.2} and $B(\cd)$ and $\D(\cd)$ are locally integrable $[\bH]$-valued functions on $\cI$ such that $B(t)=B^*(t)$ and $\D(t)\geq 0$ (a.e. on $\cI$).

A  function $y\in\AC$ is a solution of system \eqref{3.3} if equality \eqref{3.3} holds a.e. on $\cI$. An operator function $Y(\cd,\l):\cI\to [\cK,\bH]$ is an operator solution of \eqref{3.3} if $y(t)=Y(t,\l)h$ is a solution of \eqref{3.3} for every $h\in\cK$ (here $\cK$ is a Hilbert space with $\dim\cK<\infty$).

In the sequel we denote by $\cN_\l,\; \l\in\bC,$ the linear space of all solutions of the system \eqref{3.3} belonging to $\lI$. According to \cite{KogRof75,LesMal03} the numbers $N_\pm=\dim\cN_\l,\;\l\in\bC_\pm,$ do not depend on $\l$ in either $\bC_+$ or $\bC_-$. These numbers are called the formal deficiency indices of the system \cite{KogRof75}. Clearly $N_\pm\leq n$.

In the following for each operator solution $Y(\cd,\l)\in \lo{\cK}$ we denote by $Y(\l)$ the linear operator from $\cK$ to $\cN_\l$ given by $(Y(\l) h)(t)=Y(t,\l)h, \; h\in\cK$.

Clearly, for any $\l\in\bC$ the space $\cN$  of all solutions $y$ of \eqref{3.3} with $\D(t)y(t)=0$ (a.e. on $\cI$) is a subspace of $\cN_\l$; moreover, $\cN$  does depend on $\l$. The space $\cN$ is called the  null manifold  of the system \cite{KogRof75}. Denote by $\t_\cN$ the subspace in $\bH$ given by
\begin{gather}\label{3.5}
\t_\cN=\{y(a):y\in\cN\}.
\end{gather}
As is known \cite{Orc, Kac03, LesMal03} system
\eqref{3.3} gives rise to the maximal linear relations
$\tma$ and $\Tma$  in  $\lI$ and $\LI$ respectively. They are
given by
\begin {equation*}
\tma=\{\{y,f\}\in(\lI)^2 :y\in\AC \;\;\text{and}\;\; J y'(t)-B(t)y(t)=\D(t)
f(t)\;\;\text{a.e. on}\;\; \cI \}
\end{equation*}
and $\Tma=\wt\pi_\D\tma$. Moreover the
Lagrange's identity
\begin {equation}\label{3.6}
(f,z)_\D-(y,g)_\D=[y,z]_b - (J y(a),z(a)),\quad \{y,f\}, \;
\{z,g\} \in\tma
\end{equation}
holds with
\begin {equation*}
[y,z]_b:=\lim_{t \uparrow b}(J y(t),z(t)), \quad y,z \in\dom\tma.
\end{equation*}
Let $\cD_b$ be the set of all $y\in\dom \tma$ such that $[y,z]_b=0$ for all $ z\in\dom\tma$.
The minimal relation $\Tmi$ in $\LI$ is defined via $\Tmi=\wt\pi_\D\cT_a$, where
\begin {equation}\label{3.6.0}
\cT_a=\{\{y,f\}\in\tma: y\in\cD_b, \;y(a)=0\}.
\end{equation}
As was shown in \cite{Orc,Kac03,LesMal03,Mog12} $\Tmi$ is a closed symmetric linear relation in $\LI$, $\Tmi^*=\Tma$ and
\begin{gather}\label{3.6.1}
n_+(\Tmi)=N_+ - \dim\cN,\qquad n_-(\Tmi)= N_- -\dim\cN.
\end{gather}
With each subspace $\t\subset \bH$ we associate the subspace $\t^\tm\subset \bH$ given by
\begin {equation*}
\t^\tm=\bH\ominus J\t=\{h\in\bH:(Jh,k)=0,\; k\in \t\}.
\end{equation*}
Clearly $\t^{\tm\tm}=\t$. Moreover, by \cite[Proposition 4.19]{Mog12}
\begin{gather}\label{3.7}
\t_\cN^\tm=\{y(a):y\in\cD_b\}.
\end{gather}
Denote by $Sym (\bH)$ the set of all subspaces $\t$ in $\bH$ satisfying $\t\subset \t^\tm$ or, equivalently, $(J h, k)=0, \; h,k\in \t$.

The following three lemmas will be useful in the sequel.
\begin{lemma}\label{lem3.1.1}
{\rm(1)} If $\eta\in Sym (\bH)$, then $\dim \eta\leq\nu$ and $\dim \eta^\tm\geq \nu+\wh \nu$.

{\rm(2)} For every $\eta\in Sym (\bH)$ there exists a subspace $\t\subset\bH$ such that $\t^\tm\in Sym (\bH)$, $\dim\t=\nu+\wh\nu$ (i.e., the dimension of $\t$ is minimally possible) and $\t^\tm\cap \eta=\{0\}$.

{\rm(3)} Let $\t$ be a subspace in $\bH$ and $\t^\tm\in Sym (\bH)$. Then there exist an operator $\wt U\in [\bH]$ and a subspace $H_1\subset H$ such that $\wt U^*J\wt U=J$ and $\wt U \bH_0=\t$,  where
\begin{gather}\label{3.7.1}
\bH_0=H\oplus\wh H\oplus H_1.
\end{gather}
\end{lemma}
\begin{proof}
(1) Let $\wh J$ and $X$ be operators in $\bH$ given by the block representations
\begin{gather*}
\wh J=i \begin{pmatrix} I_H &  0 & 0  \cr 0 & I_{\wh H}  & 0 \cr  0& 0& -I_H \end{pmatrix}, \qquad X = \tfrac 1 {\sqrt 2} \begin{pmatrix} -i I_H & 0 &  I_H\cr 0 & \sqrt 2 I_{\wh H} & 0 \cr i I_H & 0 &  I_H \end{pmatrix}
\end{gather*}
with respect to decomposition \eqref{3.1} of $\bH$. One can easily verify  that
\begin{gather}\label{3.7.1a}
X^*\wh J X=J, \qquad X^* X = X X^*=I_\bH.
\end{gather}
and the equality ${\rm gr}\, V_\eta= X\eta$ gives a bijective correspondence between all $\eta\in Sym (\bH)$ and all isometries $V_\eta\in [\dom V_\eta, H]$  with $\dom V_\eta \subset H\oplus \wh H $. Hence for every $\eta\in Sym (\bH)$ one has  $\dim \eta=\dim\ran V_\eta\leq\nu$ and, consequently, $\dim\eta^\tm\geq \nu+\wh\nu$.

(2) Assume that $\eta\in Sym (\bH)$ and let $U\in [\dom U,H]$ be an isometry such that $\dom U\in H\oplus\wh H$, $-V_\eta\subset U$ and $\ran U=H$. Then $U=U_{\t_0}$ with some $\t_0\in Sym (\bH)$ and the obvious equality ${\rm gr}\, V_\eta\cap {\rm gr}\,U=\{0\}$ yields $\eta\cap \t_0=\{0\}$. Moreover, $\dim \t_0=\dim\ran U=\nu$ and hence $\t:=\t_0^\tm$ possesses the required  properties.

(3) Let $H_1$ be a subspace in $H$ with $ \codim H_1=\dim \t^\tm$, let $H_1^\perp=H\ominus H_1$ and let $\bH_0\subset \bH$ be subspace \eqref{3.7.1}. Then $\bH_0^\tm=H_1^\perp\oplus \{0\}\oplus\{0\}$ and therefore $\bH_0^\tm\in Sym (\bH)$. Let $V_1=V_{\bH_0^\tm}$ and $V_2=V_{\t^\tm}$.  Since $\dim\bH_0^\tm=\dim\t^\tm$, one has $\dim\dom V_1=\dim\dom V_2$. Therefore there exist unitary operators $U_1\in [H\oplus\wh H]$ and $U_2\in [H]$ such that $U_1 \dom V_1= \dom V_2$ and $V_2U_1\up \dom V_1=U_2 V_1$. Letting $\wh U={\rm diag} (U_1,U_2)$ one obtains the operator $\wh U \in [\bH]$ such that $\wh U ^*\wh J \wh U=\wh J$ and $\wh U {\rm gr}\, V_1={\rm gr}\, V_2$. This and \eqref{3.7.1a} imply that the operator $\wt U:=X^*\wh U X$ satisfies $\wt U^*J\wt U=J$ and $\wt U \bH_0^\tm=\t^\tm$. Therefore  $\wt U \bH_0=\t$.
\end{proof}
\begin{remark}\label{rem3.1.2}
If $H_1\subset H$ is a subspace from Lemma \ref{lem3.1.1} (3), $H_1^\perp=H\ominus H_1$ and $\bH_0$ is given by \eqref{3.7.1}, then the following decompositions are obvious:
\begin{gather}\label{3.7.2}
\bH_0=\underbrace{H_1^\perp\oplus H_1}_{H}\oplus\wh H\oplus H_1, \qquad \bH = \underbrace{H_1^\perp\oplus H_1}_{H}\oplus \wh H\oplus \underbrace{H_1\oplus H_1^\perp}_{H}=\bH_0\oplus H_1^\perp.
\end{gather}
\end{remark}
\begin{lemma}\label{lem3.2}
Let $\t$ be a subspace in $\bH$. Then:

{\rm (1)} The equalities
\begin{gather}
T=T_{\t^\tm}:=\{\wt\pi_\D\{y,f\}:\{y,f\}\in\tma,\; y\in\cD_b,\; y(a)\in \t^\tm\}\label{3.8}\\
T^*=\{\wt\pi_\D\{y,f\}:\{y,f\}\in\tma,\; y(a)\in \t\}\label{3.9}
\end{gather}
defines a relation $T\in \C (\LI)$ and its adjoint $T^*$. Moreover,  $\Tmi\subset T\subset \Tma $

{\rm (2)} The multivalued part $\mul T$ of $T$  is the set of all $\wt f\in\gH$ such that for some (and hence for all) $f(\cd)\in\wt f$ there exists a solution $y$ of the system
\begin{gather}\label{3.9.1}
Jy'-B(t)y=\D(t) f(t), \quad t\in \cI
\end{gather}
satisfying $\D(t)y(t)=0\;\; ({\rm a.e.\;\; on}\;\; \cI)$, $y(a)\in \t^\tm$ and $y\in\cD_b$.

{\rm (3)} The relation $T$ is symmetric if and only if $\t^\tm\cap \t_\cN^\tm\in Sym (\bH)$.
\end{lemma}
\begin{proof}
(1) The inclusions $\Tmi\subset T\subset \Tma $ directly follow from \eqref{3.8} and definitions of $\Tmi$ and $\Tma$.  Next we show that the relation $T^*$ adjoint to $T$ is of the form \eqref{3.9}. In  view of the Lagrange's identity \eqref{3.6} for every $\{y,f\}\in \tma$ with $y(a)\in\t$ one has $\wt\pi_\D \{y,f\}\in T^*$. Conversely, assume that $\{\wt y,\wt f\}\in T^*$ and prove the existence of $\{y,f\}\in \tma$ such that $y(a)\in \t$ and $\wt\pi_\D \{y,f\}=\{\wt y,\wt f\}$. Since $\Tmi\subset T$, it follows that $T^*\subset \Tma$ and hence there is $\{y_1,f\}\in \tma$ such that $\wt\pi_\D \{y_1,f\}=\{\wt y,\wt f\}$. Let $h\in  \t^\tm\cap \t_\cN^\tm$. Then in view of \eqref{3.7} there exists $\{z,g\}\in \tma$ such that $z\in\cD_b, \; z(a)=h$ and hence $\{\wt z,\wt g\}:=\wt\pi_\D \{z,g\} \in T$. Applying the Lagrange's identity \eqref{3.6} to $ \{y_1,f\}$ and $\{z,g\}$ one obtains
\begin{gather*}
(Jy_1(a),h)= (y_1,g)_\D -(f,z)_\D=(\wt y,\wt g)-(\wt f,\wt z)=0, \quad h\in  \t^\tm\cap \t_\cN^\tm.
\end{gather*}
Therefore $y_1(a)\in (\t^\tm\cap \t_\cN^\tm)^\tm$. Obviously $(\t^\tm\cap \t_\cN^\tm)^\tm=\t +\t_\cN$ and hence $y_1(a)=h+y_2(a)$ with some $h\in\t$ and $y_2\in\cN$. Let $y=y_1-y_2$. Since $\{y_2,0\}\in\tma$, it follows that a pair $\{y,f\}:=\{y_1,f\}-\{y_2,0\}$ belongs to $\tma$. Moreover, $y(a)=y_1(a) - y_2(a)=h$ and hence $y(a)\in \t$. Finally, $\pi_\D y_2=0$ and therefore $\wt\pi_\D \{y,f\}=\wt\pi_\D \{y_1,f\}= \{\wt y,\wt f\}$. This completes the proof of \eqref{3.9}.

 Statement (2) directly follows from \eqref{2.1}.

(3) It follows from  \eqref{3.7} that $T=T_{\t^\tm\cap \t_\cN^\tm}$. Therefore  to prove statement (3) it is sufficient to prove the following   equivalent  statement: if $\t^\tm\subset \t_\cN^\tm$, then the equivalence $T\subset T^* \iff \t^\tm\subset \t$ is valid.

So assume that $\t^\tm\subset \t_\cN^\tm$ and let $T\subset T^* $. If $h,k\in \t^\tm$, then by \eqref{3.7} there exist $\{y,f\}, \{z,g\}\in \tma$ such that $y,z\in\cD_b$ and $y(a)=h,\; z(a)=k$. Therefore $\wt\pi_\D\{y,f\}$,  $ \wt\pi_\D\{z,g\}\in T$ and hence
\begin{gather*}
(f,z)_\D-(y,g)_\D=0.
\end{gather*}
This and the Lagrange's identity \eqref{3.6} imply that $(Jh,k)=0$. Therefore $\t^\tm\subset \t$. If conversely $\t^\tm\subset \t$, then the inclusion $T\subset T^* $ directly follows from \eqref{3.8} and \eqref{3.9}.
\end{proof}
\begin{lemma}\label{lem3.3}
There exists a subspace $\t\subset \bH$ such that $\t^\tm \in Sym (\bH)$, $\dim\t=\nu+\wh\nu$ and the symmetric extension $T=T_{\t^\tm}$ of $\Tmi$ defined by \eqref{3.8} satisfies $\mul T=\mul \Tmi$.
\end{lemma}
\begin{proof}
Let $\eta$ be a subspace in $\bH$ defined by
\begin{gather}\label{3.10}
\eta = \{y(a): y\in\cD_b, \; \D(t)y(t)=0\; (\text {a.e. on} \;\;\cI) \}.
\end{gather}
If $h,k\in \eta$, then there exist $\{y,f\}, \{z,g\}\in\tma$ such that $y,z\in\cD_b$, $y(a)=h$, $z(a)=k$ and $\D(t)y(t)=\D(t)z(t)=0$ (a.e. on $\cI$). Application of the Lagrange's  identity \eqref{3.6} to such $\{y,f\}$ and $\{z,g\}$ gives $(Jh,k)=0$, which implies that $\eta\in Sym (\bH)$. Therefore by Lemma \ref{lem3.1.1}, (2) there exists a subspace $\t\subset \bH$ such that $\t^\tm \in Sym (\bH), \; \dim \t=\nu+\wh \nu$ and $\t^\tm\cap\eta =\{0\}$. Let $T=T_{\t^\tm}$ be given by \eqref{3.8} and let $\wt f\in\mul T$. Then according to Lemma \ref{lem3.2}, (2) there exists $y\in\cD_b$ such that $y(a)\in\t^\tm$, $\D(t)y(t)=0$ (a.e. on $\cI$) and $\{y,f\}\in \tma$ for each $f(\cd)\in\wt f$. Since by \eqref{3.10} $y(a)\in \t^\tm\cap\eta$, it follows that $y(a)=0$ and hence $\{y,f\}\in\cT_a$. Hence $\{\pi_\D y, \wt f\}\in \Tmi$ and the equality $\pi_\D y=0$ yields $\wt f\in \mul\Tmi$. Thus $\mul T\subset \mul\Tmi$ and in view of the obvious inclusion $\mul\Tmi\subset\mul T$ one has $\mul T=\mul \Tmi$.
\end{proof}
\subsection{$q$-pseudospectral and spectral functions}
In what follows we put $\gH:=\LI$ and denote by $\gH_b$ the set of all $\wt f\in\gH$ with the following property: there exists $\b_{\wt f}\in\cI$ such that for some (and hence for all) function $f\in\wt f$ the equality $\D(t)f(t)=0$ holds a.e. on $(\b_{\wt f}, b)$.

Let $\t$ and $\bH_0'$ be  subspaces in $\bH$, let $K=K_\t\in [\bH_0',\bH]$ be an operator such that $\ker K_\t=\{0\}$ and $K_\t\bH_0'=\t$ and let $\f_K(\cd,\l)(\in [\bH_0',\bH])$ be an operator solution of \eqref{3.3} satisfying $\f_K(a,\l)=K,\; \l\in\bC$. With each $\wt f\in\gH_b$ we associate the function $\wh f(\cd):\bR\to \bH_0'$ given by
\begin{equation}\label{3.11}
\wh f(s)=\int_\cI \f_K^*(t,s)\D(t)f(t)\,dt,\quad f(\cd)\in\wt f.
\end{equation}
One can easily prove that $\wh f(\cd)$ is a continuous  function on $\bR$.

Recall that an operator $V\in [\gH_1,\gH_2]$ is  a partial isometry if $||Vf||=||f||$ for all $f\in\gH_1\ominus\ker V$.
\begin{definition}\label{def3.4}
A distribution function $\s(\cd):\bR\to [\bH_0']$ is called a $q$-pseudospectral function of the system \eqref {3.3} (with respect  to the operator $K= K_\t$) if $\wh f\in\lSa$ for all $\wt f\in\gH_b$ and the operator $V \wt f:=\pi_\s\wh f, \; \wt f\in\gH_b,$ admits a continuation to a partial isometry $V_\s\in [\gH,\LSa]$. The operator $V_\s$ is  called the (generalized) Fourier transform corresponding to $\s(\cd)$.
\end{definition}
Clearly, if $\s(\cd)$ is a $q$-pseudospectral function, then for each $f(\cd)\in\lI$  there exists a unique (up to the seminorm in $\lSa$) function $\wh f(\cd)\in\lSa$ such that
\begin{equation*}
\lim\limits_{\b\uparrow b}\Bigl |\Bigl |\wh f(\cd)-\int\limits_{[a,\b)} \f_K^*(t,\cd)\D(t)f(t)\,dt\Bigr|\Bigr|_{\lSa}=0.
\end{equation*}
The function $\wh f(\cd)$ is called the Fourier transform of the function $f(\cd)$.
\begin{remark}\label{rem3.5}
Similarly to \cite{DunSch,Sht57}(see also \cite[Proposition 3.4]{Mog15}) one proves that for each $q$-pseudospectral function $\s(\cd)$
\begin{equation}\label{3.12}
V_\s^*\wt g=\pi_\D \left( \int_\bR \f_K(\cd,s)\, d \s(s) g(s)\right),\quad \wt g\in \LSa, \;\; g(\cd)\in\wt g,
\end{equation}
where the integral converges in the seminorm of $\lI$. Hence for each function $f(\cd)\in \lI$ with $\pi_\D f\in \gH\ominus \ker V_\s$ the equality (the inverse Fourier transform)
\begin{equation*}
f(t)=\int_\bR \f_K(t,s)\, d \s(s) \wh f(s)
\end{equation*}
is valid. Therefore the natural problem is a characterization of $q$-pseudospectral functions $\s(\cd)$ with the minimally possible kernel of $V_\s$.
\end{remark}
The following lemma can be proved in the same way as Lemma 3.7 in \cite{Mog15.2}.
\begin{lemma}\label{lem3.6}
Assume that $\t$ and $\bH_0'$ are subspaces in $\bH$, $\s(\cd)$ is a $q$-pseudospectral function (with respect to $K_\t\in [\bH_0',\bH]$), $V_\s$ is the corresponding Fourier transform  and $T\in \C (\gH)$ is given by \eqref{3.8}. Then there exist a Hilbert space $\wt\gH\supset \gH$ and a self-adjoint operator $\wt T_0$ in $\wt\gH_0:=\wt\gH\ominus \ker V_\s$ such that $\wt T_0\subset T_{\wt\gH}^*$ (here $T_{\wt\gH}^*\in \C (\wt\gH)$ is the linear relation adjoint to $T$ in $\wt\gH$).
\end{lemma}
By using Lemma \ref{lem3.6} one can prove similarly to \cite[Proposition 3.8]{Mog15.2} the following theorem.
\begin{theorem}\label{th3.7}
Let the assumptions of Lemma \ref{lem3.6} be satisfied and let $\mul T$ be the multivalued part of $T$  (see Lemma \ref{lem3.2}, (2)). Then
\begin {equation}\label{3.14}
\mul T\subset \ker V_\s
\end{equation}
\end{theorem}
\begin{definition}\label{def3.8}
Under the assumptions of Theorem \ref{th3.7} a $q$-pseudospectral function $\s(\cd)$ of the system \eqref{3.3} with $\ker V_\s=\mul T$  is called a pseudospectral function .
\end{definition}
\begin{definition}\label{def3.9}
Let $\t$ and $\bH_0'$ be subspaces in $\bH$. A distribution function $\s(\cd):\bR\to [\bH_0']$ is called a  spectral function of the system \eqref{3.3} (with respect to $K_\t\in [\bH_0',\bH]$) if for every $\wt f\in\gH_b$ the inclusion $\wh f\in \lSa$ holds and the Parseval equality $||\wh f||_{\lSa}=||\wt f||_\gH$ is valid (for $\wh f$ see \eqref{3.11}).

The number $n_\s:=\dim \bH_0'(=\dim\t)$ is called a dimension of the spectral function $\s(\cd)$.
\end{definition}
If for a given $K_\t\in [\bH_0',\bH]$ a pseudospectral function $\s(\cd)$ exists, then in view of  \eqref{3.14} it is a $q$-pseudospectral function with the minimally possible $\ker V_\s$ (see the problem posted  in Remark \ref{rem3.5}). Moreover, \eqref{3.14} yields the following proposition.
\begin{proposition}\label{pr3.10}
Let $\t$ and $\bH_0'$ be subspaces in $\bH$ and let $T\in \C (\gH)$ be given by \eqref{3.8}. If $\mul T\neq \{0\}$, then the set of spectral functions (with respect to $K_\t\in [\bH_0',\bH]$) is empty. If $\mul T = \{0\}$, then the sets of spectral and pseudospectral functions (with respect to $K_\t$) coincide.
\end{proposition}
A connection between different $q$-pseudospectral functions corresponding to the same subspace $\t\subset \bH$ is specified in the following proposition.
\begin{proposition}\label{pr3.10.1}
Assume that $\t$ and $\bH_{0j}'$ are subspaces in $\bH$ and $K_j=K_{j\t}\in [\bH_{0j}',\bH]$ are operators such that $\ker K_j=\{0\}$ and $K_j \bH_{0j}'=\t, \; j\in \{1,2\}$. Then: (1) there exists a unique isomorphism $X\in [\bH_{01}',\bH_{02}' ]$ such that $K_1=K_2 X$; (2) the equality $\s_2(s)=X \s_1(s)X^*$ gives a bijective correspondence between all $q$-pseudospectral functions $\s_1(\cd)$ (with respect to $K_1$) and $\s_2(\cd)$  (with respect to  $K_2$) of the system \eqref{3.3}. Moreover $\s_2(\cd)$ is a pseudospectral or spectral function if and only if so is $\s_1(\cd)$.
\end{proposition}
\begin{proof}
Statement (1) is obvious To prove statement (2) assume that $\s_1(\cd)$ is a $q$-pseudo\-spec\-tral function (with respect to $K_1$) and $\s_2(\cd)$ is an $[\bH_{02}' ]$ -valued distribution function given by $\s_2(s)=X \s_1(s)X^*$. One can easily verify that the equality $(\cU g)(s)=X^{-1*}g(s), \; g= g(\cd)\in \cL^2(\s_1; \bH_{01}'),$ defines a surjective linear operator $\cU:\cL^2(\s_1; \bH_{01}') \to \cL^2(\s_2; \bH_{02}')$ satisfying $||\cU g||_{\cL^2(\s_2; \bH_{02}')}=||g||_{\cL^2(\s_1; \bH_{01}')}$. Therefore the equality $U\wt g=\pi_{\s_2}\cU g,\; \wt g\in L^2(\s_1; \bH_{01}'), g\in \wt g,$ defines a unitary operator $U\in [L^2(\s_1; \bH_{01}'), L^2(\s_2; \bH_{02}')]$. Next assume that $\wt f\in\gH_b, \; f(\cd)\in\wt f$ and $\wh f_j(\cd)$ is the Fourier transform of $f(\cd)$ given by \eqref{3.11} with $\f_K(\cd,\l)=\f_{K_j}(\cd,\l), \; j\in\{1,2\}$. Since obviously $\f_{K_1}(t,s)=\f_{K_2}(t,s)X$, it follows that $\wh f_2(s)=X^{-1*}\wh f_1(s)$. Hence $\wh f_2=\cU \wh f_1\in \cL^2(\s_2; \bH_{02}')$ and $\pi_{\s_2}\wh f_2= U \pi_{\s_1}\wh f_1 = U V_{\s_1}\wt f$. This implies that the operator $V_2 \wt f:=\pi_{\s_2}\wh f_2,\; \wt f\in\gH_b,$ admits a continuation to the partial isometry $V_{\s_2}=U V_{\s_1}(\in [\gH, L^2(\s_2; \bH_{02}')])$ with $\ker V_{\s_2}=\ker V_{\s_1}$. Therefore $\s_2(\cd)$ is a $q$-pseudospectral function (with respect to $K_2)$; moreover, $\s_2(\cd)$ is a pseudospectral or spectral function if and only if so is $\s_1(\cd)$.
\end{proof}
\begin{remark}\label{rem3.10.2}
It follows from Proposition \ref{pr3.10.1} that a $q$-pseudospectral (in particular pseudospectral) function $\s(\cd)$ with respect to the operator  $K_\t\in [\bH_0',\bH]$ is uniquely characterized  by the subspace $\t\subset \bH$.
\end{remark}

Under  the assumptions of Theorem \ref{th3.7}  we let $\gH_0:=\gH\ominus \mul T$, so that
\begin {equation*}
\gH=\mul T\oplus \gH_0.
\end{equation*}
Moreover, for a pseudospectral function $\s(\cd)$ we denote by $V_0=V_{0,\s}$ the isometry from $\gH_0$ to $\LSa$ given by $V_{0,\s}:=$ $V_\s\up \gH_0$. Next assume that $\wt\gH\supset \gH$ is a Hilbert space and  $\wt T=\wt T^*\in \C (\wt\gH)$ with $\mul \wt T=\mul T$. In the following we put  $\wt\gH_0:=\wt\gH\ominus\mul T$, so that $\gH_0\subset \wt\gH_0$ and
\begin{equation*}
\wt\gH=\mul T\oplus \wt\gH_0.
\end{equation*}
Denote also by $\wt T_0$ the operator part of $\wt T$. Clearly, $\wt T_0$ is a self-adjoint operator in $\wt\gH_0$.
\begin{proposition}\label{pr3.11}
Assume that $\t$ and $\bH_0'$ are subspaces in $\bH$, $\s(\cd)$ is a pseudospectral function (with respect to $K_\t\in [\bH_0',\bH]$)  and $T\in \C (\gH)$ is given by \eqref{3.8}. Moreover, let $L_0=V_\s\gH$ and let $\L_\s=\L_\s^*$ be   the multiplication operator  in $\LSa$ defined by
\begin{gather*}
\dom \Lambda_\s=\{\wt f\in \LSa: s f(s)\in \lSa \;\;\text{for some (and hence for all)}\;\; f(\cd)\in \wt f\}\nonumber\\
\Lambda_\s \wt f=\pi_{\s}(sf(s)), \;\; \wt f\in\dom\Lambda_\s,\quad f(\cd)\in\wt f.
\end{gather*}
Then $T$ is a symmetric extension of $\Tmi$ and there  exist a Hilbert space $\wt \gH\supset\gH$ and an exit space self-adjoint extension $\wt T\in\C (\wt\gH)$ of $T$ such that $\mul\wt T=\mul T$ and the relative spectral function $F(t)=P_{\wt\gH,\gH}E_{\wt T}((-\infty,t))\up \gH$ of $T$ satisfies
\begin {equation} \label{3.16}
((F(\b)-F(\a))\wt f,\wt f)=\int_{[\a,\b)} (d\s(s)\wh f(s),\wh f(s)), \;\;\;\;\wt f\in \gH_b, \;\;\;\; -\infty<\a<\b<\infty.
\end{equation}
Moreover,   there exists a unitary operator $\wt V \in [\wt \gH_0,\LSa]$ such that $\wt V\up \gH_0=V_{0,\s}$ and the operators $\wt T_0$ and $\L_\s$  are unitarily equivalent by means of $\wt V$.

If in addition the operator $\L_\s$ is $L_0$-minimal, then the extension $\wt T$ is unique(up to the equivalence) and $\wt T\in \wt {\rm Self}_0 (T)$ (that is, $\wt T$ is $\gH$-minimal).
\end{proposition}
\begin{proof}
By using Lemma \ref{lem3.6} one proves as in \cite[Proposition 5.6]{Mog15.2} the following statement:

(S) There is a Hilbert space $\wt\gH\supset \gH$ and a relation $\wt T=\wt T^*\in \C (\wt \gH)$ such that $\mul\wt T=\mul T$, $T\subset \wt T$ and \eqref{3.16} holds with $F(t)=P_{\wt\gH,\gH}E_{\wt T}((-\infty,t))\up \gH$.

Moreover, by Lemma \ref{lem3.2}, (1) $\Tmi\subset T$. Therefore $T$ is a symmetric extension $\Tmi$ and $F(\cd)$ is a spectral function of $T$. Other statements of the proposition can be proved as in \cite[Proposition 5.6]{Mog15.2}.
\end{proof}
\begin{definition}\label{def3.11.1}
$\,$\cite{Atk,GK} System \eqref{3.3} is called definite if $\cN=\{0\}$ or, equivalently, if for some (and hence for all) $\l\in\bC$ there exists only a trivial solution $y=0$ of this system satisfying $\D(t)y(t)=0$ (a.e. on $\cI$).
\end{definition}
\begin{proposition}\label{pr3.12}
Let $\t$ be a subspaces in $\bH$ and let  $\s(\cd)$ be a pseudospectral function (with respect to $K_\t\in [\bH_0',\bH]$). Then $\t^\tm\cap\t_\cN^\tm\in Sym  (\bH)$. If in addition the system is definite, then  $\t^\tm\in Sym  (\bH)$.
\end{proposition}
\begin{proof}
The first statement is immediate  from  Proposition \ref{pr3.11} and Lemma \ref{lem3.2}, (3). For a definite system $\t_\cN=\{0\}$ and hence $\t_\cN^\tm=\bH$. This yields the second statement.
\end{proof}
\begin{remark}\label{rem3.13}
Proposition \ref{pr3.12} shows that a necessary condition for existence of a pseudospectral function for a given $\t$  is $\t^\tm\cap\t_\cN^\tm\in Sym  (\bH)$. Clearly this condition is satisfied if $\t^\tm \in Sym  (\bH)$.
\end{remark}
\section{$m$-functions  of symmetric  systems}
\subsection{Boundary pairs and boundary triplets for symmetric systems}\label{sect4.1}
\begin{definition}\label{def4.1}
Let $\t$ be a subspace in $\bH$. System \eqref{3.3} will be called $\t$-definite if the conditions $y\in\cN$ and $y(a)\in\t$  yield $y=0$.
\end{definition}
\begin{remark}\label{rem4.2}
If system is definite then obviously it is $\t$-definite for any $\t\in\bH$. Hence $\t$-definiteness is  generally speaking a weaker condition then definiteness. At the same time in the case $\t=\bH(\Leftrightarrow \t^\tm=\{0\})$ definiteness of the system is the
same as $\t$-definiteness.
\end{remark}
The following assertion directly follows from definition of $\Tmi$ and  \eqref{3.9},  \eqref{2.1}.
\begin{assertion}\label{ass4.2.0}
{\rm (1)} The equality $\mul \Tmi=\{0\}$ is equivalent to the following condition:

\rm{(C0)} If $f(\cd)\in\lI$ and there exists a solution $y(\cd)$ of \eqref{3.9.1} such that $\D(t)y(t)=0\;\; ({\rm a.e.\;\; on}\;\; \cI)$, $y(a)=0$ and $y\in\cD_b$, then $\D(t) f(t)=0$ (a.e. on $\cI$).

{\rm (2)} Let $\t^\tm\in Sym(\bH)$, let system \eqref{3.3} be $\t$-definite and let $T$ be the relation \eqref{3.8}.  Then  the equalities $\mul T=\{0\}$, $\mul T=\mul T^*$ and $\mul T^*=\{0\}$ are equivalent to the following conditions \rm{(C1)}, \rm{(C2)} and \rm{(C3)} respectively:

\rm{(C1)} If $f(\cd)\in\lI$ and there exists a solution $y(\cd)$ of the system \eqref{3.9.1} such that $\D(t)y(t)=0\;\; ({\rm a.e.\;\; on}\;\; \cI)$, $y(a)\in \t^\tm$ and $y\in\cD_b$, then $\D(t) f(t)=0$ (a.e. on $\cI$).

\rm{(C2)} If $f(\cd)\in\lI$ and $y(\cd)$  is  a solution of \eqref{3.9.1} such that $y(a)\in\t$ and $\D(t) y(t)=0$ (a.e. on $\cI$), then $y(\cd)\in \cD_b$ and $y(a)\in\t^\tm$.

\rm{(C3)}  If $f(\cd)\in\lI$ and there exists a solution $y(\cd)$ of \eqref{3.9.1} satisfying $\D(t) y(t)=0$ (a.e. on $\cI$) and $y(a)\in\t$, then $\D(t) f(t)=0$ (a.e. on $\cI$).
\end{assertion}
The following proposition can be proved in the same way as Proposition 5.5 in \cite{Mog15.2}.
\begin{proposition}\label{pr4.2.1}
Assume that $\t$ and $\bH_0'$ are subspaces in $\bH$, $\s(\cd)$ is a $q$-pseudospectral function (with respect to $K_\t\in [\bH_0',\bH]$)and $L_0:=V_\s \gH$. If system is $\t$-definite, then the multiplication  operator $\L_\s$ is $L_0$-minimal.
\end{proposition}
Below within this section we suppose the following assumptions:

{\bf (A1)}  $\t$ is a subspace in $\bH$ and $\t^\tm\in Sym (\bH)$. Moreover, system \eqref{3.3} is $\t$-definite and satisfies $N_-\leq N_+$.

{\bf (A2)} $H_1$ is a subspace  in $H$, $\bH_0\subset \bH$ is the subspace \eqref{3.7.1}, $\wt U\in [\bH]$ is an operator satisfying $\wt U^* J\wt U =J$ and $\wt U\bH_0=\t$, $\G_a:\dom\tma\to \bH$ is the linear operator given by $\G_a y = \wt U^{-1}y(a), \; y\in\dom\tma,$ and
\begin{gather}\label{4.1}
\G_a=(\G_{0a}^1,\, \G_{0a}^2,\, \wh \G_a,\, \G_{1a}^2,\, \G_{1a}^1)^\top:\dom\tma\to H_1^\perp\oplus H_1\oplus \wh H\oplus H_1\oplus H_1^\perp
\end{gather}
is the block representation of $\G_a$ in accordance with the decomposition \eqref{3.7.2} of $\bH$.

{\bf (A3)} $\wt \cH_b$ and $\cH_b\subset \wt \cH_b$  are finite dimensional  Hilbert  spaces and
\begin{gather}\label{4.2}
\G_b=(\G_{0b},\,  \wh\G_b,\,  \G_{1b})^\top:\dom\tma\to
\wt \cH_b\oplus\wh H \oplus \cH_b
\end{gather}
is a surjective linear operator satisfying for all $y,z\in\dom\tma$ the following identity
\begin {equation} \label{4.3}
 [y,z]_b=(\G_{0b}y,\G_{1b}z)-(\G_{1b}y,\G_{0b}z)
+i(P_{\cH_b^\perp}\G_{0b}y, P_{\cH_b^\perp}\G_{0b}z)+ i (\wh\G_b y, \wh\G_b z)
\end{equation}
(here $\cH_b^\perp=\wt\cH_b\ominus \cH_b$).
\begin{remark}\label{rem4.3}
Existence of the operators  $\wt U$ in assumption (A2) and $\G_b$ in assumption (A3) follows from Lemma \ref{lem3.1.1}, (3) and  \cite[Lemma 3.4]{AlbMalMog13} respectively. Moreover, in the case $N_+=N_-$ (and only in this case) one has $\wt \cH_b=\cH_b$ and the identity \eqref{4.3} takes the form
\begin{gather*}
[y,z]_b=(\G_{0b}y,\G_{1b}z)-(\G_{1b}y,\G_{0b}z)
+ i (\wh\G_b y, \wh\G_b z), \quad y,z\in\dom\tma.
\end{gather*}
Observe also that $\G_b y$ is a singular boundary value of a function $y\in\dom\tma$ at the endpoint $b$ (for more details see \cite[Remark 3.5]{AlbMalMog13}).
\end{remark}
The following lemma directly follows from  definition of the operator $\G_a$ and its block representation \eqref{4.1}.
\begin{lemma}\label{lem4.3.1}
Let $Y(\cd,\l)\in\lo{\cK}$  be an operator solution of \eqref{3.3}. Then
\begin{gather}\label{4.3.1}
\wt U^{-1}Y(a,\l)=\G_a Y(\l)=(P_{\bH,\bH_0}\G_a Y(\l), \G_{1a}^1 Y(\l))^\top:\cK\to \bH_0\oplus H_1^\perp,
\end{gather}
where
\begin{gather}\label{4.3.2}
P_{\bH,\bH_0}\G_a =(\G_{0a}^1, \G_{0a}^2, \wh\G_a,\G_{1a}^2)^\top:\dom\tma\to H_1^\perp\oplus H_1\oplus\wh H\oplus H_1.
\end{gather}
\end{lemma}
\begin{proposition}\label{pr4.4}
Let $\cH_0$ and $\cH_1\subset \cH_0$ be finite dimensional Hilbert spaces and let $\G_j':\dom\tma\to \cH_j,\; j\in\{0,1\},$ be linear operators given by
\begin{gather}
\cH_0=H_1^\perp\oplus H_1\oplus \wh H\oplus \wt\cH_b,\quad
\cH_1=H_1^\perp\oplus H_1\oplus \wh H\oplus \cH_b\label{4.4}\\
\G_0'=(-\G_{1a}^1,\,-\G_{1a}^2, \, i(\wh\G_a-\wh\G_b), \, \G_{0b} )^\top:\dom\tma\to H_1^\perp\oplus H_1\oplus \wh H\oplus \wt\cH_b\label{4.5}\\
\G_1'=(\G_{0a}^1,\,\G_{0a}^2, \, \tfrac 1 2 (\wh\G_a+\wh\G_b), \, -\G_{1b} )^\top:\dom\tma\to H_1^\perp\oplus H_1\oplus \wh H\oplus \cH_b\label{4.6}.
\end{gather}
Then $\dim\cH_0=N_+,\;\dim\cH_1=N_-$ and a pair $\bpa$ with a linear relation $\BR$ defined by
\begin{gather}\label{4.7}
\G=\{\{\wt\pi_\D \{y,f\}, \G_0'y\oplus\G_1'y\}:\{y,f\}\in\tma\}
\end{gather}
is a boundary pair for $\Tma$ such that $\cK_\G=\{0\}$ (for $\cK_\G$ see \eqref{2.21.1}).
\end{proposition}
\begin{proof}
The fact that $\bpa$ is a boundary pair for $\Tma$ as well as the equalities $\dim\cH_0=N_+,\;\dim\cH_1=N_-$ directly follow  from \cite[Theorem 5.3]{Mog12}. Next, according to \cite{Mog12} $\mul\G=\{\{\G_0'y,\G_1'y\}:y\in\cN\}$ and hence $\cK_\G=\{\G_1'y:y\in\cN \;\; \text{and}\;\;\G_0'y=0 \}$. Moreover, the equalities $\wt U^{-1}\t=\bH_0$ and \eqref{3.7.2}, \eqref{4.1} yield the equivalence
\begin{gather}\label{4.7.1}
y(a)\in\t\iff  \G_{1a}^1y=0,\quad y\in\dom\tma.
\end{gather}
Since the system is $\t$-definite, this implies the equality $\cK_\G=\{0\}$.
\end{proof}
\begin{definition}\label{def4.4.1}
The boundary pair $\bpa$ constructed in Proposition \ref{pr4.4} is called a
decomposing  boundary pair  for $\Tma$.
\end{definition}
Let $\dot\cH_0$ and $\dot\cH_1\subset \dot\cH_0$ be finite dimensional Hilbert spaces and $\dot\G_j':\dom\tma\to \dcH_j,\; j\in\{0,1\},$ be linear operators given by
\begin{gather}
\dcH_0=H_1\oplus \wh H\oplus \wt\cH_b,\quad
\dcH_1= H_1\oplus \wh H\oplus \cH_b\label{4.8}\\
\dG_0'=(-\G_{1a}^2, \, i(\wh\G_a-\wh\G_b), \, \G_{0b} )^\top:\dom\tma\to H_1\oplus \wh H\oplus \wt\cH_b\label{4.9}\\
\dG_1'=(\G_{0a}^2, \, \tfrac 1 2 (\wh\G_a+\wh\G_b), \, -\G_{1b} )^\top:\dom\tma\to  H_1\oplus \wh H\oplus \cH_b\label{4.10}.
\end{gather}
It follows from \eqref{4.4} - \eqref{4.6} that
\begin{gather}
\cH_0=H_1^\perp\oplus\dcH_0,\qquad \cH_1=H_1^\perp\oplus\dcH_1\label{4.10.1}\\
\G_0'=(-\G_{1a}^1,\,\dG_0')^\top:\dom\tma\to H_1^\perp\oplus \dcH_0\label{4.10.2}\\
\G_1'=(\G_{0a}^1,\, \dG_1')^\top:\dom\tma\to H_1^\perp\oplus \dcH_1 \label{4.10.3}.
\end{gather}
\begin{proposition}\label{pr4.5}
Let $T\in \C(\gH)$ be given by \eqref{3.8}. Then:

{\rm (1)} $T$ is a symmetric extension of $\Tmi$ and the following equalities hold:
\begin{gather}
T=\{\wt\pi_\D\{y,f\}:\{y,f\}\in\tma,\; y\in\cD_b,\;\G_{1a}^1y=0, \; \G_{0a}^2y=\G_{1a}^2 y =0, \; \wh\G_a y =0 \}\label{4.11}\\
T^*=\{\wt\pi_\D\{y,f\}:\{y,f\}\in\tma,\;\G_{1a}^1y=0\}\label{4.12}
\end{gather}

{\rm (2)} For every $\{\wt y, \wt f\}\in T^*$ there exists a unique $y\in\dom\tma$ such that $\G_{1a}^1 y=0,\; \pi_\D y=\wt y$ and $\{y,f\}\in\tma$ for any $f\in\wt f$.

{\rm (3)} The collection $\dot\Pi=\{\dcH_0\oplus \dcH_1,
\dot\G_0, \dot\G_1\}$ with operators $\dot\G_j:T^*\to \dcH_j$ of the form
\begin {equation}\label {4.13}
\dot\G_0 \{\wt y,\wt f\}=\dot\G_{0}'y, \quad \dot\G_1 \{\wt y,\wt f\}=\dot\G_{1}'y, \quad
\{\wt y,\wt f\}\in T^*
\end{equation}
is a boundary triplet for $T^*$. In \eqref{4.13} $y\in\dom\tma$ is uniquely defined by $\{\wt y,\wt f\}$ in accordance with statement {\rm (2)}.
\end{proposition}
\begin{proof}
(1) Since $\wt U^{-1}\t=\bH_0$ and $\wt U^{-1}\t^{\tm}=\bH_0^\tm= H_1^\perp\oplus\{0\}\oplus\{0\}\oplus\{0\}$, the equivalences \eqref{4.7.1} and
\begin {equation*}
y(a)\in \t^\tm \iff (\G_{1a}^1 y=0,\;\G_{0a}^2y=\G_{1a}^2 y =0, \; \wh\G_a y =0 ), \quad y\in\dom\tma
\end{equation*}
are valid. This and \eqref{3.8}, \eqref{3.9} yield \eqref{4.11} and \eqref{4.12}.

By using $\t$-definiteness of the system one proves statement (2) similarly to \cite[Proposition 4.5, (2)]{Mog15.2}.

(3) Equalities \eqref{4.10.2}, \eqref{4.10.3} and  identity \eqref{2.21} for the decomposing boundary pair  yield the Green's identity  \eqref{2.14.3} for operators $\dG_0$ and $\dG_1$. To prove surjectivity of the operator $(\dG_0, \dG_1)^\top$ it is sufficient to show that
\begin{gather}\label{4.13.1}
\ker\dG_0\cap \ker\dG_1 =T, \qquad \dim\dcH_0=n_+(T), \qquad \dim\dcH_1=n_-(T).
\end{gather}
Clearly, $\{\wt y,\wt f\}\in \ker\dG_0\cap \ker\dG_1$ if and only if there is $\{y,f\}\in\tma$ such  that $\wt\pi_\D\{y,f\}=\{\wt y,\wt f\} $ and  $\G_{1a}^1 y=0$, $\G_{0a}^2y=\G_{1a}^2 y =0, \; \wh\G_a y =0 $, $\G_b y=0$. Moreover, in view of \eqref{4.3} and surjectivity of the operator $\G_b$ the equivalence $\G_b y=0\iff y\in\cD_b$ is valid. This yields the first equality in \eqref{4.13.1}. Next assume that
\begin {equation*}
\cT=\{\{y,f\}\in\tma: y\in\cD_b, \;y(a)\in \t^\tm\cap \t_\cN^\tm\}.
\end{equation*}
It follows from \eqref{3.7} and \eqref{3.6.0} that $\dim (\dom \cT / \dom\cT_a )=\dim (\t^\tm\cap \t_\cN^\tm)$ and $T=\wt\pi_\D \cT$. If $\{y,f\}\in \cT$ and $\wt\pi_\D \{y,f\}=0$, then $y\in\cN$ and $y(a)\in\t^\tm\subset\t$. Therefore in view of $\t$-definiteness $y=0$ and consequently $\ker \wt\pi_\D\up\cT=\{0\}$. This and the obvious equality $\dim (\cT / \cT_a) =\dim (\dom \cT / \dom\cT_a )$ imply that
\begin{gather}\label{4.13.2}
\dim (T/\Tmi)=\dim (\t^\tm\cap \t_\cN^\tm).
\end{gather}
In view of $\t$-definiteness one has $\t\cap \t_\cN =\{0\}$. Since obviously $\t^\tm\cap \t_\cN^\tm=(\t \dotplus \t_N)^\tm$, it follows that
\begin{gather}\label{4.13.3}
\dim (\t^\tm\cap \t_\cN^\tm)=n-\dim \t -\dim \t_\cN=\codim \t-\dim \cN.
\end{gather}
Combining \eqref{4.13.2} and \eqref{4.13.3} with the well known equality $n_\pm (T)=n_\pm (\Tmi)- \dim (T/\Tmi)$ and taking \eqref{3.6.1} into account on gets $n_\pm (T)=N_\pm-\codim \t$. Moreover, the equality $\wt U^{-1}\t=\bH_0$ yields $\codim \t=\dim H_1^\perp$ and according to Proposition \ref{pr4.4}  $\dim\cH_0=N_+,\;\dim\cH_1=N_-$. This implies that
\begin{gather}\label{4.14}
n_+ (T)= \dim\cH_0-\dim H_1^\perp, \quad n_- (T)= \dim\cH_1-\dim H_1^\perp
\end{gather}
Now combining \eqref{4.14} with \eqref{4.10.1} one obtains the second and third equalities in \eqref{4.13.1}.
\end{proof}
\subsection{$\cL_\D^2$-solutions of boundary problems}
\begin{definition}\label{def4.6}
Let $\dcH_0$ and $\dcH_1$ be given by \eqref{4.8}. A boundary parameter is a pair
\begin{gather}\label{4.15}
\tau=\tau(\l)=\{C_0(\l),C_1(\l) \}\in\wt R(\dcH_0,\dcH_1),\quad \l\in\bC_+,
\end{gather}
where $C_j(\l)(\in  [\dcH_j,\dcH_0]), \; j\in\{0,1\},$ are holomorphic operator functions satisfying \eqref{2.14.1}.
\end{definition}
In the case $N_+=N_-$ (and only in this case) $\wt\cH_b=\cH_b, \; \dcH_0=\dcH_1=:\dcH$ and $\tau\in \wt R(\dcH)$. If in addition $\tau=\tau(\l) \in \wt R^0(\dcH)$ is an operator pair \eqref{2.14.2}, then a boundary parameter $\tau$ will be called self-adjoint.

Let $\tau$ be a boundary parameter \eqref{4.15}. For a given $f\in\lI$ consider the  boundary value problem
\begin{gather}
J y'-B(t)y=\l \D(t)y+\D(t)f(t), \quad t\in\cI\label{4.16}\\
\G_{1a}^1 y=0, \quad
C_0(\l)\dG_{0}'y -C_1(\l)\dG_{1}'y=0, \quad \l\in\bC_+ \label{4.17}
\end{gather}
A function $y(\cd,\cd):\cI\tm \bC_+ \to \bH$ is called a solution of this problem if for each $\l\in\bC_+$ the function $y(\cd,\l)$ belongs to $\AC\cap\lI$ and
satisfies the equation \eqref{4.16} a.e. on $\cI$ (so that $y\in\dom\tma$) and
the boundary conditions \eqref{4.17}.

The following theorem is a consequence of Theorem 3.11 in \cite{Mog13.2} applied to the boundary triplet $\dot\Pi$ for $T^*$.
\begin{theorem}\label{th4.7}
Let  under the assumptions {\rm (A1) - (A3)} $T$ be a symmetric relation  \eqref{3.8} (or equivalently \eqref{4.11}). If $\tau$
is a boundary parameter \eqref{4.15}, then for every $f\in\lI$ the problem \eqref{4.16}, \eqref{4.17} has a unique solution
$y(t,\l)=y_f(t,\l) $ and the equality
\begin {equation*}
R(\l)\wt f = \pi_\D(y_f(\cd,\l)), \quad \wt f\in \gH, \quad f\in\wt f, \quad
\l\in\bC_+
\end{equation*}
defines a generalized resolvent $R(\l)=:R_\tau(\l)$ of $T$. Conversely, for
each generalized resolvent $R(\l)$ of $T$ there exists a unique boundary
parameter $\tau$ such that $R(\l)=R_\tau(\l)$. Moreover, if $N_+=N_-$, then  $R_\tau(\l)$ is a
canonical resolvent if and only if $\tau$ is a self-adjoint
boundary parameter \eqref{2.14.2}. In this case $R_\tau(\l)=(\wt T_\tau-\l)^{-1}$, where $\wt T_\tau\in {\rm Self}(T)$ is given by
\begin{gather}\label{4.17.1}
\wt T_\tau=\{\wt\pi_\D\{y,f\}:\{y,f\}\in\tma, \G_{1a}^1y=0,\, C_0\dG_0'y-C_1\dG_1'y=0\}.
\end{gather}
\end{theorem}
\begin{proposition}\label{pr4.9}
For any $\l\in\CR$ there exists a unique collection of operator solutions $\xi_1(\cd,\l) \in\lo{H_1^\perp}$,  $\xi_2(\cd,\l) \in\lo{H_1}$ and $\xi_3(\cd,\l) \in\lo{\wh H}$ of the system \eqref{3.3} satisfying the boundary conditions
\begin{gather}
\G_{1a}^1 \xi_1(\l)=-I_{H_1^\perp}, \quad \G_{1a}^2 \xi_1(\l)=0, \quad \wh\G_a\xi_1(\l)=\wh\G_b\xi_1(\l),\;\;\;\; \l\in\CR \label{4.18}\\
\G_{0b} \xi_1(\l)=0, \;\;\;\;\l\in\bC_+;\qquad P_{\wt\cH_b,\cH_b}\G_{0b} \xi_1(\l)=0,
\;\;\;\;\l\in\bC_-. \label{4.19}\\
\G_{1a}^1 \xi_2(\l)=0, \quad \G_{1a}^2 \xi_2(\l)=-I_{H_1}, \quad \wh\G_a\xi_2(\l)=\wh\G_b\xi_2(\l),\;\;\;\; \l\in\CR \label{4.20}\\
\G_{0b} \xi_2(\l)=0, \;\;\;\;\l\in\bC_+;\qquad P_{\wt\cH_b,\cH_b}\G_{0b} \xi_2(\l)=0,
\;\;\;\;\l\in\bC_-. \label{4.21}\\
\G_{1a}^1 \xi_3(\l)=0, \quad \G_{1a}^2 \xi_3(\l)=0, \quad i(\wh\G_a- \wh\G_b)\xi_3(\l)=I_{\wh H},\;\;\;\; \l\in\CR \label{4.22}\\
\G_{0b} \xi_3(\l)=0, \;\;\;\;\l\in\bC_+;\qquad P_{\wt\cH_b,\cH_b}\G_{0b} \xi_3(\l)=0,
\;\;\;\;\l\in\bC_-. \label{4.23}
\end{gather}
Moreover, for any $\l\in\bC_+ \; (\l\in\bC_-)$  there exists a unique operator solution $u_+(\cd,\l)\in\lo{\wt\cH_b}$ (resp. $u_-(\cd,\l)\in\lo{\cH_b}$) satisfying the boundary conditions
\begin{gather}
\G_{1a}^1u_\pm(\l)=0,\quad \G_{1a}^2u_\pm(\l)=0, \quad \wh\G_a u_\pm(\l)=\wh\G_b u_\pm(\l),  \quad
\l\in\bC_\pm\label{4.24}\\
\G_{0b}u_+(\l)=I_{\wt\cH_b},\;\; \l\in\bC_+; \qquad
P_{\wt\cH_b,\cH_b}\G_{0b}u_-(\l)=I_{\cH_b},\;\; \l\in\bC_-.\label{4.25}
\end{gather}
\end{proposition}
\begin{proof}
Let $\bpa$ be the decomposing boundary pair \eqref{4.7} for $\Tma$.  Then the linear relation $\G_0=P_{\cH_0\oplus\{0\}}\G:\gH^2\to \cH_0$ for this triplet is \begin {equation}\label{4.26}
\G_0=\{\{\wt\pi_\D\{y,f\},\G_0'y\}:\{y,f\}\in \tma\}.
\end{equation}
By using \eqref{4.26} one proves in the same way as in \cite[Proposition 4.8]{Mog15.2} that
\begin {equation}\label{4.27}
\G_0\up\wh\gN_\l(\Tmi)=\{\{\wt\pi_\D\{y,\l y\},\G_0'y\}:\,y\in \cN_\l\}, \quad \l\in\CR.
\end{equation}
Since by  Proposition \ref{pr4.4} $\cK_\G=\{0\}$, it follows from Proposition \ref{pr3.15} that

\noindent $\ran \G_0\up\wh\gN_\l(\Tmi)=\cH_0$ and \eqref{4.27} yields $\G_0' \cN_\l=\cH_0, \; \l\in\bC_+$. Moreover, by Proposition \ref{4.4} $\dim\cN_\l=\dim\cH_0$ and hence for each $\l\in\bC_+$ the operator $\G_0'\up\cN_\l$ isomorphically maps $\cN_\l$ onto $\cH_0$. Similarly by using   \eqref{4.27} one  proves that for each $\l\in\bC_-$ the operator $P_1\G_0'\up\cN_\l$ isomorphically maps $\cN_\l$ onto $\cH_1$. Therefore the equalities  $Z_+(\l)=(\G_0'\up\cN_\l)^{-1}, \; \l\in\bC_+,$ and $Z_-(\l)=(P_1\G_0'\up\cN_\l)^{-1}, \; \l\in\bC_-,$ define the isomorphisms $Z_+(\l):\cH_0\to\cN_\l$ and $Z_-(\l):\cH_1\to\cN_\l$ such that
\begin {equation}\label{4.28}
\G_0' Z_+(\l)=I_{\cH_0}, \;\;\l\in\bC_+; \qquad P_1\G_0' Z_-(\l)=I_{\cH_1}, \;\;\l\in\bC_-.
\end{equation}
Assume that the block representations of $Z_\pm(\l)$ are
\begin {gather}
Z_+(\l)=(\xi_1(\l),\,\xi_2(\l),\,\xi_3(\l),\, u_+(\l)  ):H_1^\perp\oplus H_1\oplus \wh H \oplus \wt\cH_b\to
\cN_\l,\;\;\l\in\bC_+\label{4.29}\\
Z_-(\l)=(\xi_1(\l),\,\xi_2(\l),\,\xi_3(\l),\, u_-(\l)  ):H_1^\perp\oplus H_1\oplus \wh H \oplus \cH_b\to
\cN_\l,\;\;\l\in\bC_-\label{4.30}
\end{gather}
and let $\xi_1(\cd,\l) \in\lo{H_1^\perp}$,  $\xi_2(\cd,\l) \in\lo{H_1}$ , $\xi_3(\cd,\l) \in\lo{\wh H}$, $u_+(\cd,\l)\in\lo{\wt\cH_b}$  and $u_-(\cd,\l)\in\lo{\cH_b}$ be the respective operator solutions of \eqref{3.3}. It follows from \eqref{4.5} that
\begin {gather}\label{4.31}
P_1\G_0'=(-\G_{1a}^1,\,-\G_{1a}^2, \, i(\wh\G_a-\wh\G_b), \, P_{\wt\cH_b,\cH_b}\G_{0b} )^\top:\dom\tma\to H_1^\perp\oplus H_1\oplus \wh H\oplus \cH_b
\end{gather}
Now combining \eqref{4.28} with \eqref{4.5}, \eqref{4.29}, \eqref{4.31}, \eqref{4.30} and taking the block representations of $I_{\cH_0}$ and $I_{\cH_1}$ into account one gets the equalities \eqref{4.18} - \eqref{4.25}. Finally, uniqueness of specified operator solutions is implied by the equalities $\ker \G_0'\up\cN_\l=\{0\}, \; \l\in\bC_+,$ and $\ker P_1\G_0'\up\cN_\l=\{0\}, \; \l\in\bC_-$.
\end{proof}
\begin{proposition}\label{pr4.10}
The Weyl function $M_+=M_+(\l),\; \l\in\bC_+,$ of the decomposing boundary pair $\bpa$ for $\Tma$ admits the block representation
\begin{gather}\label{4.32}
M_+=\begin{pmatrix} M_{11} & M_{12}& M_{13}& M_{14}\cr M_{21} & M_{22}& M_{23}& M_{24}\cr M_{31} & M_{32}& M_{33}& M_{34}\cr M_{41} & M_{42}& M_{43}& M_{44}\end{pmatrix}: \underbrace{H_1^\perp\oplus H_1 \oplus \wh H\oplus\wt\cH_b}_{\cH_0}\to \underbrace{H_1^\perp\oplus H_1 \oplus \wh H\oplus\cH_b}_{\cH_1},
\end{gather}
with entries  $M_{jk}=M_{jk}(\l), \;\l\in\bC_+,$ defined by
\begin{gather}
M_{jk}(\l)=\G_{0a}^j\xi_k(\l), \;\; j\in\{1,2\}, \; k\in\{1,2,3\};\quad M_{j4}(\l)=\G_{0a}^j u_+(\l), \;\; j\in\{1,2\} \label{4.33}\\
M_{3k}(\l)=\wh\G_a \xi_k(\l), \; k\in\{1,2\};\;\; M_{33}(\l)=\wh \G_{a}\xi_3(\l)+\tfrac i 2 I_{\wh H}, \;\; M_{34}(\l)=\wh\G_{a} u_+(\l)\label{4.34}\\
M_{4k}(\l)=-\G_{1b}\xi_k(\l),\;\;k\in\{1,2,3\}; \quad M_{44}(\l)=-\G_{1b}u_+(\l)\label{4.35}.
\end{gather}
Moreover, for every $\l\in\bC_-$ one has
\begin{gather}
M_{jk}^*(\ov\l)=\G_{0a}^k\xi_j(\l), \;\;k\in\{1,2\},\; j\in\{1,2,3\}; \;\; M_{4k}^*(\ov\l)=\G_{0a}^k u_-(\l), \;\; k\in\{1,2\} \label{4.36}\\
M_{j3}^*(\ov\l)=\wh \G_{a}\xi_j(\l),\;j\in\{1,2\};\;\;\; M_{33}^*(\ov\l)=\wh \G_{a}\xi_3(\l)+\tfrac i 2 I_{\wh H}, \;\;M_{43}^*(\ov\l)=\wh \G_{a} u_-(\l)\label{4.37}
\end{gather}
\end{proposition}
\begin{proof}
Let $Z_\pm(\l)$ be the same as in the proof of Proposition \ref{pr4.9}. Then by \eqref{4.7}
\begin{gather*}
\{\pi_\D Z_+(\l)h_0\oplus \l\pi_\D Z_+(\l)h_0, \G_0'Z_+(\l)h_0\oplus \G_1'Z_+(\l)h_0 \}\in\G,\quad h_0\in\cH_0,\;\; \l\in\bC_+\\
\{\pi_\D Z_-(\l)h_1\oplus \l\pi_\D Z_-(\l)h_1, \G_0'Z_-(\l)h_1\oplus \G_1'Z_-(\l)h_1 \}\in\G,\quad h_1\in\cH_1,\;\; \l\in\bC_-
\end{gather*}
and in view of \eqref{2.22} and \eqref{2.23} one has
\begin{gather*}
\G_1'Z_+(\l)=M_+(\l)\G_0'Z_+(\l); \qquad  (\G_1'+iP_2 \G_0')Z_-(\l)=M_+^*(\ov\l)P_1\G_0'Z_-(\l).
\end{gather*}
(the first equality holds for $\l\in\bC_+$, while the second one for $\l\in\bC_-$). This and \eqref{4.28} imply that
\begin{equation}\label{4.38}
\G_1' Z_+(\l)=M_+(\l), \;\;\l\in\bC_+; \qquad   (\G_1'+i P_2\G_0')
Z_-(\l)=M_+^*(\ov\l), \;\; \l\in\bC_-.
\end{equation}
It follows from \eqref{4.5} and \eqref{4.6} that
\begin{multline}\label{4.39}
\G_1'+i P_2\G_0'=(\G_{0a}^1,\,\G_{0a}^2, \, \tfrac 1 2 (\wh\G_a+\wh\G_b), \, *)^\top:\dom\tma\to H_1^\perp\oplus H_1\oplus \wh H\oplus \wt\cH_b
\end{multline}
(the entry $*$ does not matter). Assume that \eqref{4.32} is the block representation of $M_+(\l)$. Combining the  first equality in \eqref{4.38} with \eqref{4.6}, \eqref{4.29} and taking the last equalities in \eqref{4.18}, \eqref{4.20},  \eqref{4.22} and \eqref{4.24} into account one gets \eqref{4.33} - \eqref{4.35}. Similarly combining the second equality in \eqref{4.38} with \eqref{4.39} and \eqref{4.30} one obtains \eqref{4.36} and \eqref{4.37}.
\end{proof}
Using the entries $M_{ij}=M_{ij}(\l)$ from the block representation \eqref{4.32} of $M_+(\l)$ introduce the holomorphic operator-functions $m_0=m_0(\l)(\in [\bH_0]), \;  S_1=S_1(\l)(\in [\dcH_0,\bH_0]), \; S_2=S_2(\l)(\in [\bH_0,\dcH_1])$ and $\dot M_+=M_+(\l)(\in [\dcH_0,\dcH_1]),\;\l\in\bC_+,$ by setting
\begin{gather}
m_0=\begin{pmatrix} M_{11} & M_{12}& M_{13}& 0\cr M_{21} & M_{22}& M_{23}& -\tfrac 1 2 I_{H_1}\cr M_{31} & M_{32}& M_{33}& 0\cr 0 & -\tfrac 1 2 I_{H_1}& 0& 0\end{pmatrix}: \underbrace{H_1^\perp\oplus H_1 \oplus \wh H\oplus H_1}_{\bH_0}\to \underbrace{H_1^\perp\oplus H_1 \oplus \wh H\oplus H_1}_{\bH_0}\label{4.40}\\
S_1=\begin{pmatrix} M_{12}& M_{13}& M_{14}\cr  M_{22}& M_{23}& M_{24}\cr  M_{32}& M_{33}-\tfrac i 2 I_{\wh H}& M_{34}\cr -I_{H_1}& 0& 0\end{pmatrix}: \underbrace{ H_1 \oplus \wh H\oplus\wt\cH_b}_{\dcH_0}\to\underbrace{H_1^\perp\oplus H_1 \oplus \wh H\oplus H_1}_{\bH_0}\label{4.41}\\
S_2=\begin{pmatrix} M_{21} & M_{22}& M_{23}& -I_{H_1}\cr M_{31} & M_{32}& M_{33}+\tfrac i 2 I_{\wh H}& 0\cr M_{41} & M_{42}& M_{43}& 0\end{pmatrix}: \underbrace{H_1^\perp\oplus H_1 \oplus \wh H\oplus H_1}_{\bH_0}\to \underbrace{ H_1 \oplus \wh H\oplus\cH_b}_{\dcH_1}\label{4.42}\\
\dot M_+=\begin{pmatrix}  M_{22}& M_{23}& M_{24}\cr  M_{32}& M_{33}& M_{34}\cr  M_{42}& M_{43}& M_{44}\end{pmatrix}: \underbrace{ H_1 \oplus \wh H\oplus\wt\cH_b}_{\dcH_0}\to \underbrace{H_1 \oplus \wh H\oplus\cH_b}_{\dcH_1}\label{4.43}
\end{gather}
\begin{lemma}\label{lem4.11}
Let $\dot\Pi=\{\dcH_0\oplus\dcH_1,\dG_0,\dG_1\}$ be the boundary triplet \eqref{4.13} for $T^*$. Moreover, let $\dot Z_+(\cd,\l)\in\lo{\dcH_0}$ and $\dot Z_-(\cd,\l)\in\lo{\dcH_1}$ be operator solutions of \eqref{3.3} given by
\begin{gather}
\dot Z_+(t,\l)=(\xi_2(t,\l), \xi_3(t,\l), u_+(t,\l) ):H_1\oplus \wh H\oplus \wt\cH_b\to\bH,\;\;\l\in\bC_+\label{4.44}\\
\dot Z_-(t,\l)=(\xi_2(t,\l), \xi_3(t,\l), u_-(t,\l) ):H_1\oplus \wh H\oplus \cH_b\to\bH,\;\;\l\in\bC_-\label{4.45}
\end{gather}
and let $\dot M_+(\cd)$ be the operator-function \eqref{4.43}. Then:

{\rm (1)} The following equalities hold
\begin{gather}
\wt U^{-1}\dot Z_+(a,\l)=\begin{pmatrix} P_{\bH,\bH_0}\G_a \dot Z_+(\l)\cr \G_{1a}^1 \dot Z_+(\l)\end{pmatrix}=\begin{pmatrix} S_1(\l)\cr 0 \end{pmatrix}:\dcH_0\to \bH_0\oplus H_1^\perp,\;\;\l\in\bC_+\label{4.45.1}\\
\wt U^{-1}\dot Z_-(a,\l)=\begin{pmatrix} P_{\bH,\bH_0}\G_a \dot Z_-(\l)\cr \G_{1a}^1 \dot Z_-(\l)\end{pmatrix}=\begin{pmatrix} S_2^*(\ov\l)\cr 0 \end{pmatrix}:\dcH_1\to \bH_0\oplus H_1^\perp,\;\;\l\in\bC_-.\label{4.45.2}
\end{gather}

{\rm(2)} $\g$-fields $\dot\g_\pm(\cd)$ of the triplet $\dot\Pi$ are
\begin{gather}\label{4.46}
\dot\g_+(\l)=\pi_\D \dot Z_+(\l),\;\;\l\in\bC_+;\qquad \dot\g_-(\l)=\pi_\D \dot Z_-(\l),\;\;\l\in\bC_-
\end{gather}
and the Weyl function of $\dot\Pi$ coincides with $\dot M_+(\l)$.

{\rm (3)} If $\tau$ is a boundary parameter \eqref{4.15}, then $(C_0(\l)-C_1(\l)\dot M_+(\l))^{-1}\in [\dcH_0]$  and
\begin{gather}\label{4.46.1}
-(\tau(\l)+\dot M_+(\l))^{-1}=(C_0(\l)-C_1(\l)\dot M_+(\l))^{-1}C_1(\l), \quad \l\in\bC_+.
\end{gather}
\end{lemma}
\begin{proof}
(1) It follows from \eqref{4.3.2} and Propositions  \ref{pr4.9},  \ref{pr4.10} that
\begin{gather}\label{4.46.2}
P_{\bH,\bH_0}\G_a \dot Z_+(\l)=S_1(\l), \qquad \G_{1a}^1 \dot Z_+(\l)=0,\quad \l\in\bC_+,
\end{gather}
and $P_{\bH,\bH_0}\G_a \dot Z_-(\l)=S_2^*(\ov\l),$ $\G_{1a}^1 \dot Z_-(\l)=0,\;\l\in\bC_-$. This and Lemma \ref{lem4.3.1} yield \eqref{4.45.1} and \eqref{4.45.2}.

(2) Let $\dot\g_\pm(\l)$ be given by \eqref{4.46} and let $Z_\pm(\l)$ be the same as in the proof of Proposition \ref{pr4.9}. Comparing \eqref{4.44} and \eqref{4.45} with \eqref{4.29} and \eqref{4.30} one gets $\dot Z_+(\l)=$

\noindent $Z_+(\l)\up\dcH_0, \; \l\in\bC_+,$ and $\dot Z_-(\l)=Z_-(\l)\up\dcH_1,\;  \l\in\bC_-$. Therefore by \eqref{4.28}
\begin{gather}\label{4.47}
\G_0'\dot Z_+(\l)h_0=h_0,\;\; h_0\in\dcH_0,\;\l\in\bC_+; \qquad P_1\G_0'\dot Z_-(\l)h_1=h_1,\;\; h_1\in\dcH_1,\;\l\in\bC_-
\end{gather}
and in view of \eqref{4.10.2} $P_1\G_0'=(-\G_{1a}^1,  \dot P_1\dG_0')^\top$ with $\dot P_1:=P_{\dcH_0,\dcH_1}$. This and \eqref{4.10.1}, \eqref{4.10.2} imply that
\begin{gather}
\G_{1a}^1 \dot Z_+(\l)=0,\quad \l\in\bC_+;\qquad \G_{1a}^1 \dot Z_-(\l)=0,\quad \l\in\bC_-\label{4.47.1}\\
\dot\G_0'\dot Z_+(\l)=I_{\dot\cH_0},\;\; \l\in\bC_+; \qquad \dot P_1\dot\G_0'\dot Z_-(\l)=I_{\dot\cH_1},\;\; \l\in\bC_-.\label{4.48}
\end{gather}
It follows from \eqref{4.47.1} that $\dot\g_+(\l)\cH_0\subset\gN_\l (T)$, $\dot\g_-(\l) \cH_1\subset\gN_\l (T)$ and \eqref{4.48} yields
\begin{gather*}
\dot\G_0\{\dot\g_+(\l)h_0,\l \dot\g_+(\l)h_0\}=\dG_0' \dot Z_+(\l)h_0=h_0,\quad h_0\in\dcH_0,\quad \l\in\bC_+\\
\dot P_1\dot\G_0\{\dot\g_-(\l)h_1,\l \dot\g_-(\l)h_1\}=\dot P_1\dG_0' \dot Z_-(\l)h_1=h_1,\quad h_1\in\dcH_1,\quad \l\in\bC_-.
\end{gather*}
Therefore according to definitions \eqref{2.16.1} and \eqref{2.16.2} $\dot\g_\pm(\cd)$ are $\g$-fields of $\dot\Pi$.

Next assume that $\dot M_+(\cd)$ is given by \eqref{4.43}. Then in view of \eqref{4.32} and \eqref{4.10.1} $\dot M_+(\l)=P_{\cH_1,\dcH_1}M_+(\l)\up\dcH_0 $ and by using \eqref{4.38} one obtains
\begin{gather}\label{4.48.1}
\dG_1' \dot Z_+(\l)=P_{\cH_1,\dcH_1}\G_1'  Z_+(\l)\up \dcH_0=P_{\cH_1,\dcH_1}M_+(\l)\up \dcH_0=\dot M_+(\l)
\end{gather}
Hence $\dot\G_1\{\dot\g_+(\l)h_0,\l \dot\g_+(\l)h_0\}=\dG_1' \dot Z_+(\l)h_0=\dot M_+(\l)h_0, \;h_0\in\dcH_0,\; \l\in\bC_+,$ and according to definition \eqref{2.16} $\dot M_+(\cd)$ is the Weyl function of $\dot\Pi$.

Statement (3) follows from \cite[Theorem 3.11]{Mog13.2}  and \cite[Lemma 2.1]{MalMog02}.
\end{proof}
\begin{theorem}\label{th4.12}
Let  $\tau$ be a boundary parameter \eqref{4.15}, let
\begin{gather}
C_0(\l)=(C_{0a}(\l), \wh C_0(\l),  C_{0b}(\l)):H_1\oplus\wh H\oplus\wt\cH_b\to\dcH_0\label{4.48.2}\\
C_1(\l)=(C_{1a}(\l), \wh C_1(\l),  C_{1b}(\l)):H_1\oplus\wh H\oplus\cH_b\to\dcH_0\label{4.48.3}
\end{gather}
be the block representations of $C_0(\l)$ and $C_1(\l)$ and let
\begin{gather}\label{4.48.4}
\Phi(\l):=(0,\,C_{0a}(\l),\, \wh C_0(\l)+\tfrac i 2 \wh C_1(\l),\, -C_{1a}(\l)): H_1^\perp\oplus H_1\oplus \wh H\oplus H_1\to\dcH_0.
\end{gather}
Then for each $\l\in\bC_+$ there exists a unique operator solution $v_\tau(\cd,\l)\in\lo{\bH_0}$ of the system \eqref{3.3} satisfying the boundary conditions
\begin{gather}\label{4.49}
\G_{1a}^1 v_\tau(\l)=-P_{\bH_0,H_1^\perp}, \qquad C_0(\l)\dG_0'v_\tau(\l)-C_1(\l)\dG_1'v_\tau(\l)=\Phi(\l)
\end{gather}
\end{theorem}
(here $P_{\bH_0,H_1^\perp}$ is the orthoprojection in $\bH_0$ onto $H_1^\perp$ in accordance with decomposition \eqref{3.7.2} of $\bH_0$).
\begin{proof}
Let $Z_0(\cd,\l)\in\lo{\bH_0}$ and $\dot Z_+(\cd,\l)\in\lo{\dcH_0}$ be operator solutions of \eqref{3.3} given by
\begin{gather}\label{4.50}
Z_0(t,\l)=(\xi_1(t,\l),\,\xi_2(t,\l), \,\xi_3(t,\l),\, 0): H_1^\perp\oplus H_1\oplus \wh H\oplus H_1 \to \bH, \quad \l\in\bC_+
\end{gather}
and \eqref{4.44} respectively and let $S_2(\l)$ be defined by \eqref{4.42}. Then in view of Lemma \ref{lem4.11},(3)
the equality
\begin {equation} \label{4.51}
v_{\tau}(t,\l)=Z_0(t,\l)+\dot Z_+(t,\l) (C_0(\l)-C_1(\l)\dot M_+(\l))^{-1}C_1(\l)S_2(\l),  \quad \l\in\bC_+
\end{equation}
correctly defines the solution $v_\tau(\cd,\l)\in\lo{\bH_0}$ of \eqref{3.3}. Let us show that  this solution satisfies \eqref{4.49}.

It follows from \eqref{4.3.2} and Propositions \ref{pr4.9}, \ref{pr4.10} that
\begin{gather}\label{4.51.2}
 P_{\bH,\bH_0}\G_a Z_0(\l)=m_0(\l)-\tfrac 1 2 J_0, \quad \G_{1a}^1 Z_0(\l)=-P_{\bH_0,H_1^\perp},\quad \l\in\bC_+,
\end{gather}
where $J_0\in [\bH_0]$ is the operator given by
\begin{gather}\label{4.52}
J_0=P_{\bH,\bH_0}J\up\bH_0=\begin{pmatrix} 0 & 0&0 &0 \cr 0&0 &0 &-I_{H_1} \cr 0&0 & iI_{\wh H}& 0 \cr 0& I_{H_1} & 0& 0 \end{pmatrix} \in [\underbrace{H_1^\perp\oplus H_1\oplus \wh H\oplus H_1}_{\bH_0}].
\end{gather}
Combining \eqref{4.51} with the second equalities in  \eqref{4.51.2} and \eqref{4.46.2} one gets the first equality in \eqref{4.49}. Next, by  \eqref{4.48} and \eqref{4.48.1}
\begin{gather*}
(C_0(\l)\dG_0' -C_1(\l)\dG_1')v_\tau(\l)=(C_0(\l)\dG_0' -C_1(\l)\dG_1') Z_0(\l)+
(C_0(\l)\dG_0' -C_1(\l)\dG_1')\tm\\
\dot Z_+(\l) (C_0(\l)-C_1(\l)\dot M_+(\l))^{-1}C_1(\l)S_2(\l)=
C_0(\l)\dG_0'Z_0(\l) +C_1(\l) (S_2(\l)-\dG_1' Z_0(\l) ).
\end{gather*}
Moreover, by \eqref{4.18} - \eqref{4.23} and \eqref{4.33} - \eqref{4.35} one has
\begin{gather*}
\dG_0'Z_0(\l)=\begin{pmatrix} 0&I_{H_1}&0&0\cr 0&0&I_{\wh H}&0 \cr 0&0&0&0 \end{pmatrix},\qquad \dG_1'Z_0(\l)=\begin{pmatrix} M_{21}(\l)&M_{21}(\l) &M_{23}(\l )& 0\cr M_{31}(\l)& M_{32}(\l) &M_{33}(\l )& 0\cr M_{41}(\l) &M_{42}(\l )&M_{43}(\l ) & 0 \end{pmatrix}
\end{gather*}
and hence  $S_2(\l)-\dG_1' Z_0(\l)=\begin{pmatrix}0&0&0&-I_{H_1}\cr 0& 0&\tfrac i 2 I_{\wh H}&0\cr0& 0&0&0  \end{pmatrix}.$ This and \eqref{4.48.2}, \eqref{4.48.3} imply that
\begin{gather*}
C_0(\l)\dG_0 v_\tau(\l)-C_1(\l)\dG_1 v_\tau(\l)= (C_{0a}(\l), \wh C_0(\l),  C_{0b}(\l)) \begin{pmatrix} 0&I_{H_1}&0&0\cr 0&0&I_{\wh H}&0 \cr 0&0&0&0 \end{pmatrix}+ \\
+(C_{1a}(\l), \wh C_1(\l),  C_{1b}(\l))\begin{pmatrix}0&0&0&-I_{H_1}\cr 0& 0&\tfrac i 2 I_{\wh H}&0\cr0& 0&0&0  \end{pmatrix}=\Phi(\l)
\end{gather*}
Thus the second equality in \eqref{4.49} is valid. Finally uniqueness of $v_\tau(\cd,\l)$ is implied by uniqueness of the solution of the problem \eqref{4.16}, \eqref{4.17} (see Theorem \ref{th4.7}).
\end{proof}
\subsection{$m$-functions}
Let $\tau$ be a boundary parameter \eqref{4.15}, let $v_\tau (\cd,\l)\in\lo{\bH_0}$ be the  operator solution of \eqref{3.3} defined in Theorem \ref{th4.12} and let $J_0$ be the operator \eqref{4.52}.
\begin{definition}\label{def4.13}
The operator function $m_\tau(\cd):\bC_+\to [\bH_0]$ defined by
\begin {equation}\label{4.53}
m_\tau(\l)=P_{\bH,\bH_0}\G_{a}v_\tau (\l)+\tfrac 1 2 J_0, \quad\l\in\bC_+
\end{equation}
is called the $m$-function  corresponding to the boundary parameter $\tau$
or, equivalently, to the boundary  value  problem \eqref{4.16}, \eqref{4.17}.
\end{definition}
In the following  theorem we provide a description of all
$m$-functions immediately in terms of the boundary parameter $\tau$.
\begin{theorem}\label{th4.14}
Let the assumptions  {\rm (A1) - (A3)} after Proposition \ref{pr4.2.1} be satisfied and let $\dcH_0$ and  $\dcH_1$ be finite-dimensional Hilbert spaces \eqref{4.8}. Assume also that $M_+(\cd)$ is the operator-function defined by \eqref{4.32} - \eqref{4.35} and $m_0(\cd), \; S_1(\cd),\; S_2(\cd)$ and $\dot M_+(\cd)$ are the operator-functions \eqref{4.40} - \eqref{4.43}. Then:

(1) $m_0(\cd)$ is the $m$-function corresponding to the boundary parameter $\tau_0=\{I_{\dcH_0}, 0_{\dcH_1,\dcH_0}\}$;

(2) for every boundary  parameter $\tau$ of the form  \eqref{4.15}  the corresponding $m$-function $m_\tau(\cd)$  admits the representation
\begin {equation}\label{4.54}
m_\tau(\l)=m_0(\l)+S_1(\l)(C_0(\l)-C_1(\l)\dot M_+(\l))^{-1}C_1(\l)
S_2(\l),\quad\l\in\bC_+.
\end{equation}
\end{theorem}
\begin{proof}
Applying the operator $P_{\bH,\bH_0}\G_a$  to the equality \eqref{4.51} and taking the first equalities in \eqref{4.51.2} and \eqref{4.46.2} into account one gets  \eqref{4.54}. Statement (1) of the theorem is immediate from \eqref{4.54}.
\end{proof}
\begin{proposition}\label{pr4.15}
The $m$-function $m_\tau(\cd)$ belongs to the class $R[\bH_0]$ and satisfies
\begin {equation}\label{4.55}
\im\, m_\tau(\l)\geq \int_\cI v_\tau^*(t,\l)\D(t)
v_\tau(t,\l)\, dt, \quad \l\in\bC_+.
\end{equation}
If $N_+=N_-$ and $\tau $ is a self-adjoint boundary parameter, then the  inequality \eqref{4.55} turns into the equality.
\end{proposition}
\begin{proof}
It follows from \eqref{4.54}   that $m_\tau(\cd)$ is holomorphic in $\bC_+$. Moreover, one can prove inequality \eqref{4.55} in the same way as similar inequalities (5.10) in \cite{AlbMalMog13} and (4.66) in \cite{Mog14}. Therefore $m_\tau(\cd)\in R[\bH_0]$.
\end{proof}
\subsection{Generalized resolvents and characteristic matrices}
In the sequel we denote by $Y_{\wt U}(\cd,\l)$ the $[\bH]$-valued operator solution of \eqref{3.3} satisfying $Y_{\wt U}(a,\l)=\wt U,\; \l\in\bC$.

The following theorem is well known (see e.g. \cite{Bru78,DLS93,Sht57}).
\begin{theorem}\label{th4.16}
For each generalized resolvent $R(\l)$ of $\Tmi$ there exists an operator-function $\Om(\cd)\in R[\bH]$ (the characteristic matrix of $R(\l)$) such that for any $\wt f\in \gH$ and $\l\in\bC_+$
\begin {equation}\label{4.56}
R(\l)\wt f=\pi_\D\left(\int_\cI Y_{\wt U}(x,\l)(\Om(\l)+\tfrac 1 2 \, {\rm sgn}(t-x)J)Y_{\wt U}^*(t,\ov\l)\D(t) f(t)\,dt \right), \;\; f\in\wt f,.
\end{equation}
\end{theorem}
\begin{proposition}\label{pr4.17}
Let $\tau$ be a boundary parameter \eqref{4.15} and let $R_\tau(\l)$ be the corresponding generalized resolvent of $T$ (and hence of $\Tmi$) in accordance with Theorem \ref{th4.7}.  Moreover, let $P_{\bH_0,H_1^\perp}$ and $I_{H_1^\perp,\bH_0}$ be the orthoprojection in $\bH_0$ onto $H_1^\perp$ and the embedding operator of $H_1^\perp$ into $\bH_0$ respectively (see decomposition \eqref{3.7.2} of $\bH_0$). Then the equality
\begin{gather}\label{4.56.1}
\Om(\l)=\begin{pmatrix}  m_\tau(\l) & -\tfrac 1 2 I_{H_1^\perp,\bH_0} \cr -\tfrac 1 2 P_{\bH_0,H_1^\perp} & 0\end{pmatrix}:\underbrace{\bH_0\oplus H_1^\perp }_{\bH}\to \underbrace{\bH_0\oplus H_1^\perp }_{\bH},\;\;\l\in\bC_+
\end{gather}
defines a characteristic matrix $\Om(\cd)$ of $R_\tau(\l)$.
\end{proposition}
\begin{proof}
Assume that $\dot\g_\pm(\cd)$ are $\g$-fields and $\dot M_+(\cd)$ is the Weyl function of the boundary triplet $\dot\Pi=\{\dcH_0\oplus \dcH_1,\dG_0,\dG_1 \}$ for $T^*$ defined in Proposition \ref{pr4.5}. Moreover let
\begin {equation}\label{4.57}
B_\tau(\l):=-(\tau(\l)+\dot M_+(\l))^{-1}=(C_0 (\l)- C_1(\l) \dot M_+(\l))^{-1}C_1(\l), \quad \l\in\bC_+
\end{equation}
(see \eqref{4.46.1}). Then according to \cite[Theorem
3.11]{Mog13.2} the  Krein type formula for generalized  resolvents
\begin {equation}\label{4.57.1}
R(\l)=R_\tau(\l)=(A_0-\l)^{-1}+\dot\g_+(\l)B_\tau(\l)\dot\g_-^*(\ov\l), \quad \l\in\bC_+
\end{equation}
holds with the maximal symmetric extension $A_0$ of $T$ given by
\begin {equation*}
A_0:=\ker \dG_0=\{\wt\pi_\D \{y,f\}: \{y,f\}\in \tma, \,\G_{1a}^1 y=0,\, \G_{1a}^2 y=0, \wh \G_a y=\wh \G_b y,\, \G_{0b}y=0 \}.
\end{equation*}
According to \cite[(4.36)]{Mog14} for each $\wt f\in\gH$ and $\l\in\bC_+$
\begin {equation}\label{4.58}
(A_0-\l)^{-1}\wt f\! =\pi_\D\left(\int_\cI Y_{\wt U}(x,\l)(\Om_0(\l)+\tfrac 1 2 \, {\rm
sgn}(t-x)J)Y_{\wt U}^*(t,\ov\l)\D(t) f(t)\,dt \right),
\end{equation}
where $f(\cd)\in\wt f$ and $\Om_0(\l)$ is the operator function defined in \cite[(4.30)]{Mog14} (actually \eqref{4.58} is proved in \cite{Mog14} for definite systems but the proof is suitable for the case of  a $\t$-definite system as well). One can easily verify that $\Om_0 (\l)$ admits the representation
\begin{gather}\label{4.59}
\Om_0(\l)=\begin{pmatrix}  m_0(\l) & -\tfrac 1 2 I_{H_1^\perp,\bH_0} \cr -\tfrac 1 2 P_{\bH_0,H_1^\perp} & 0\end{pmatrix}:\bH_0\oplus H_1^\perp \to \bH_0\oplus H_1^\perp ,\;\;\l\in\bC_+
\end{gather}
with $m_0(\l)$ given by \eqref{4.40}. Next, $\dot Z_\pm(t,\l)=Y_{\wt U}(t,\l)\wt U^{-1}\dot Z_\pm(a,\l)$ and in view of the second equality in  \eqref{4.46} and \cite[Lemma 3.3]{AlbMalMog13} one has
\begin {equation*}
\dot\g_-^*(\ov\l)\wt f=\smallint_\cI \dot Z_-^*(t,\ov\l) \D(t)f(t)\,dt=\smallint_\cI (\wt U^{-1}\dot Z_-(a,\ov\l))^* Y_{\wt U}^*(t,\ov\l)\D(t)f(t)\,dt, \quad f(\cd)\in\wt f.
\end{equation*}
This and  the first equality in  \eqref{4.46} imply that for any $\wt f\in\gH$ and $\l\in\bC_+$
\begin{gather*}
\dot\g_+(\l)B_\tau(\l)\dot\g_-^*(\ov\l)\wt f=\\
\pi_\D  \smallint_\cI Y_{\wt U}(\cd,\l)(\wt U^{-1}\dot Z_+(a,\l))
B_\tau(\l)(\wt U^{-1}\dot Z_-(a,\ov\l))^* Y_{\wt U}^*(t,\ov\l)\D(t)f(t)\,dt =\\
\pi_\D  \smallint_\cI Y_{\wt U}(\cd,\l) \ov \Om (\l)Y_{\wt U}^*(t,\ov\l) \D(t)f(t)\,dt , \quad  f(\cd)\in\wt f,
\end{gather*}
where
\begin{gather*}
\ov \Om (\l)=(\wt U^{-1}\dot Z_+(a,\l))
B_\tau(\l)(\wt U^{-1}\dot Z_-(a,\ov\l))^*=\begin{pmatrix} S_1(\l)\cr 0 \end{pmatrix} B_\tau(\l)\, (S_2(\l),\,0)=\\
\begin{pmatrix}S_1(\l) B_\tau(\l)S_2(\l) &0 \cr 0&0 \end{pmatrix}
\end{gather*}
(here we made use of  \eqref{4.45.1} and \eqref{4.45.2}).
Combining these relations  with \eqref{4.57.1} and \eqref{4.58} one obtains  the equality \eqref{4.56} with
\begin{gather*}
\Om(\l)=\Om_0(\l)+\ov \Om (\l)=\begin{pmatrix} m_0(\l)+S_1(\l) B_\tau(\l)S_2(\l) & -\tfrac 1 2 I_{H_1^\perp,\bH_0} \cr -\tfrac 1 2 P_{\bH_0,H_1^\perp}&0 \end{pmatrix}.
\end{gather*}
Hence $\Om(\l)$ is a characteristic matrix of $R_\tau(\l)$ and in view of \eqref{4.54} and \eqref{4.57} the equality \eqref{4.56.1} is valid.
\end{proof}
\section{Parametrization of pseudospectral and spectral functions}
As before we suppose in this section (unless otherwise stated)the assumptions (A1)-(A3) specified  after Proposition  \ref{pr4.2.1}.

Let $T$ be a symmetric relation \eqref{3.8}. Then according to Theorem \ref{th4.7} the boundary value problem \eqref{4.16},  \eqref{4.17} induces parametrizations  $R(\l)=R_\tau(\l)$, $\wt T=\wt T_\tau$ and $F(\cd)=F_\tau(\cd)$ of all generalized resolvents $R(\l)$, exit space extensions $\wt T\in \wt{\rm Self}(T)$ and spectral functions $F(\cd)$ of $T$ respectively by means of the boundary parameter $\tau$. Here $\wt T_\tau(\in  \wt{\rm Self}(T))$ is the extension of $T$ generating $R_\tau(\l)$ and $F_\tau(\cd)$ is the respective spectral function  of $T$.
\begin{definition}\label{def5.1}
Let  $\dot M_+=\dot M_+(\l)$ be given by  \eqref{4.43}. A boundary parameter $\tau$ of the form \eqref{4.15} is called admissible if
\begin {gather}
\lim_{y\to +\infty} \tfrac 1 {i y}P_{\dot\cH_0,\dot\cH_1}( C_0(i y)- C_1(i y )\dot M_{+}(i y))^{-1}C_1(i y)=0\label{5.1}\\
\lim_{y\to +\infty} \tfrac 1 {i y}\dot M_{+}(i y)( C_0(i y)- C_1(i y )\dot M_{+}(i y))^{-1} C_0(i y)\up{\dot\cH_1}=0\label{5.2}
\end{gather}
\end{definition}
\begin{proposition}\label{pr5.2}
 An extension $\wt T = \wt T^\tau$ belongs to ${\rm Self}_0(T)$ if and only if the boundary parameter $\tau$ is admissible. Therefore the set of admissible boundary parameters is not empty.
\end{proposition}
\begin{proof}
According to Lemma \ref{lem4.11}, (2) $\dot M_{+}(\cd)$ is the Weyl function of the boundary  triplet $\dot\Pi$ for $T^*$. Therefore the required result follows from  \cite[Theorem 2.15]{Mog15}.
\end{proof}
In the following with the operator $\wt U $ from assumption (A2) we associate the operator $U=U_\t\in [\bH_0,\bH]$ given by $U=\wt U\up \bH_0$. Moreover, we denote by $\f_U(\cd,\l)$ the $[\bH_0,\bH]$-valued operator solution of   \eqref{3.3} with $\f_U(a,\l)=U$. Clearly $\ker U=\{0\}$ and $U\bH_0=\t$.
\begin{theorem}\label{th5.3}
Let $\tau$ be an admissible boundary parameter, let $F(\cd)=F_\tau(\cd)$ be the corresponding spectral function of $T$ and let $m_\tau(\cd)$ be the $m$-function \eqref{4.53}. Then there exists a unique pseudospectral function $\s(\cd)=\s_\tau(\cd)$ of the system \eqref{3.3} (with respect to $U\in [\bH_0,\bH]$) satisfying \eqref{3.16}. This pseudospectral function is defined by the Stieltjes inversion formula
\begin {gather}\label{5.3}
\s_\tau(s)=\lim\limits_{\d\to+0}\lim\limits_{\varepsilon\to +0} \frac 1 \pi
\int_{-\d}^{s-\d}\im \,m_\tau(u+i\varepsilon)\, du.
\end{gather}
\end{theorem}
\begin{proof}
Assume that $\Om(\cd)\in R[\bH]$ is the characteristic matrix \eqref{4.56.1} of $R_\tau(\l)$ and $\Si(\cd):\bR\to [\bH]$ is the distribution function defined by
\begin {gather*}
\Si(s)=\lim\limits_{\d\to+0}\lim\limits_{\varepsilon\to +0} \frac 1 \pi
\int_{-\d}^{s-\d}\im \, \Om(u+i\varepsilon)\, du.
\end{gather*}
Using \eqref{4.56} and the Stieltjes - Liv\u{s}ic  formula one proves as in  \cite{DLS93,Sht57} the equality
\begin {equation} \label{5.5}
((F(\b)-F(\a))\wt f,\wt
f)=\int_{[\a,\b)}(d\Si(s)\wh f_0(s),\wh f_0(s)),\quad \wt f\in\gH_b,\quad -\infty<\a <\b <\infty.
\end{equation}
with the function $\wh f_0:\bR\to \bH$ defined for each $\wt f\in\gH_b$ by $\wh f_0(s)=\int\limits_\cI Y_{\wt u}^* (t,s) \D(t)f(t)\, dt, \; f(\cd)\in\wt f$. Let $\wt f\in\gH_b$, let
\begin {gather}\label{5.6}
\wh f(s)=\int_\cI \f_{u}^* (t,s) \D(t)f(t)\, dt, \quad f(\cd)\in\wt f
\end{gather}
and let $\s(\cd)=\s_\tau(\cd)$ be the distribution function \eqref{5.3}. Since $\f_{u} (t,\l)=Y_{\wt u}(t,\l)\up\bH_0$, it follows that $\wh f(s)=P_{\bH,\bH_0}\wh f_0(s)$. Moreover, by \eqref{4.56.1} one has
\begin {gather*}
\Si (s)= \begin{pmatrix} \s(s) & 0 \cr 0&0 \end{pmatrix}:\bH_0\oplus H_1^\perp \to \bH_0\oplus H_1^\perp.
\end{gather*}
This and \eqref{5.5} yield the equality \eqref{3.16}. Next by using \eqref{3.16} and Proposition \ref{pr5.2} one proves that $\s(\cd)$ is a pseudospectral function (with respect to $U$) in the same way as in \cite[Theorem 3.20]{Mog15} and \cite[Theorem 5.4]{Mog15.2}.

Let us prove that $\s(\cd)=\s_\tau(\cd)$ is a unique pseudospectral function satisfying \eqref{3.16} (we give only the sketch of the proof because it is similar to that of the alike result in \cite[Theorem 5.4]{Mog15.2}). Let $\wt\s (\cd)$ be a pseudospectral function  (with respect to $U$) such that \eqref{3.16} holds with $\wt \s(\cd)$ instead of $\s(\cd)$. Then according to \cite{DunSch} there exists a scalar measure $\mu$ on Borel sets in $\bR$ and functions $\Psi_j(\cd):\bR\to [\bH_0], \; j\in\{1,2\},$ such that
\begin {equation}\label{5.6.1}
\s(\b)-\s(\a)=\int_{\d} \Psi_1(s)\, d\mu(s), \qquad \wt\s(\b)-\wt\s(\a)=\int_{\d} \Psi_2(s)\, d\mu(s),\;\;{\d}=[ \a,\b).
\end{equation}
Let $\Psi(s):=\Psi_1(s)-\Psi_2(s)$ and let $\mu_0$  be the Lebesgue measure on Borel sets in $\cI$. Denote also by $\cG$ the set of all functions $\wh f(\cd):\bR\to\bH_0$ admitting the representation \eqref{5.6} with some $\wt f\in\gH_b$. As in \cite[Theorem 5.4]{Mog15.2} one proves that for each $\wh f\in\cG$  there is a Borel set $C_{\wh f}\subset \bR$ such that
\begin {equation}\label{5.9}
\mu(\bR\setminus C_{\wh f})=0\;\;\;\text{and}\;\;\;\mu_0 (\{t\in\cI:\D(t) \f_U(t,s)\Psi(s)\wh f(s)\neq 0\})=0,\;\;s\in C_{\wh f}.
\end{equation}
Let $s\in C_{\wh f}$ and let $y=y(t)=\f_U(t,s)\Psi(s)\wh f(s)$. Then $y$ is a solution of the system \eqref{3.3} with $\l=s$ and by \eqref{5.9} $\D(t)y(t)=0$  ($\mu_0$ a.e. on $\cI$). Hence $y\in\cN$. Moreover, $y(a)=U\Psi (s)\wh f(s)\in \t$. Since system is $\t$-definite, this implies that $y=0$ and, consequently,  $\Psi(s)\wh f(s)=0$. Thus for any $\wh f\in\cG$ there exists a Borel set $C_{\wh f}\subset\bR$ such that
\begin {equation}\label{5.10}
\mu(\bR\setminus C_{\wh f})=0\;\;\text{and}\;\;\Psi(s)\wh f(s)=0,\;\;s\in C_{\wh f}.
\end{equation}
Next we prove the following statement:

(S) for any $s\in\bR$ and $h\in \bH_0$ there is $\wh f(\cd)\in\cG$ such that $\wh f(s)=h$.

\noindent Indeed, let $s\in\bR, \; h'\in \bH_0$ and $(\wh f(s),h')=0$ for any $\wh f(\cd)\in \cG$. Put $y=y(t)=\f_U(t,s)h'$. Then
for any $\b\in\cI$ one has
$\wh f_\b(\cd):=\int\limits_{[a,\b]} \f_U^*(t,\cd)\D(t)y(t)\,dt\in\cG$ and,consequently,
\begin {equation*}
0=(\wh f_\b(s),h')= \int_{[a,\b]} (\f_U^*(t,s)\D(t)y(t),h')\,dt=\int_{[a,\b]}
(\D(t)y(t),y(t))\,dt, \quad \b\in\cI.
\end{equation*}
Hence  $y\in\cN$. Moreover, $y(a)=Uh'\in\t$ and $\t$-definiteness of the system implies that $y=0$. Therefore  $h'=0$, which proves statement (S).

Next by using \eqref{5.10} and statement (S) one proves the equality $\Psi(s)=0$ ($\mu$-a.e. on $\bR$) in the same way as in \cite[Theorem 5.4]{Mog15.2}. Thus $\Psi_1(s)=\Psi_2(s)$ ($\mu$-a.e. on $\bR$) and by \eqref{5.6.1} $\wt\s(s)=\s(s)$.
\end{proof}
\begin{corollary}\label{cor5.3.1}
{\rm (1)} Let the assumption {\rm(A1)} from Section \ref{sect4.1} be satisfied. Then the set of pseudospectral functions (with respect to $K_\t\in [\bH_0',\bH]$) is not empty.

{\rm (2)} Let system \eqref{3.3} be definite, let $N_-\leq N_+$ and let $\t$ be a subspace in $\bH$. Then the set of pseudospectral functions (with respect to $K_\t\in [\bH_0',\bH]$) is not empty if and only if $\t^\tm\in Sym (\bH)$.
\end{corollary}
\begin{proof}
Statement (1) is immediate from Proposition \ref{pr5.2} and Theorem \ref{th5.3}. Statement (2) follows from statement (1), Remark \ref{rem4.2} and Proposition \ref{pr3.12}.
\end{proof}

A parametrization of all pseudospectral functions $\s(\cd)$ (with respect to $U\in [\bH_0,\bH]$) immediately in terms of a boundary parameter $\tau$ is given by the following theorem.
\begin{theorem}\label{th5.4}
Let the assumptions be the same as in Theorem \ref{th4.14}. Then the equality
\begin {equation}\label{5.11}
m_\tau(\l)=m_0(\l)+S_1(\l)(C_0(\l)-C_1(\l)\dot M_+(\l))^{-1}C_1(\l)
S_2(\l),\quad\l\in\bC_+
\end{equation}
together with formula \eqref{5.3} establishes a bijective correspondence $\s(s)=\s_\tau(s)$ between all admissible boundary parameters $\tau$ defined by \eqref{4.15} and all pseudospectral functions $\s(\cd)$ of the system \eqref{3.3} (with respect to $U\in [\bH_0,\bH]$).
\end{theorem}
The proof of Theorem \ref{th5.4} is based on Theorems \ref{th5.3}, \ref{th4.14} and Propositions \ref{pr3.11}, \ref{pr5.2}. We omit this proof because it is similar to that of Theorem 5.7 in \cite{Mog15.2}.

The following theorem directly follows from Theorem \ref{th5.3} and Propositions \ref{pr3.11}, \ref{pr4.2.1}.
\begin{theorem}\label{th5.5}
Let the assumptions {\rm (A1)} and {\rm (A2)} from Section \ref{sect4.1} be satisfied. Then there is a one to one correspondence $\s(\cd)=\s_{\wt T}(\cd)$ between all extensions $\wt T\in \Sel$ and all pseudospectral functions $\s(\cd)$ of the system \eqref{3.3} (with respect to $U\in [\bH_0,\bH]$). This correspondence is given by the equality \eqref{3.16}, where $F(\cd)$ is a spectral function of $T$ generated by $\wt T$. Moreover, the operators $\wt T_0$ (the operator part of $\wt T$) and $\L_\s$ are unitarily equivalent and hence they have the same spectral properties. In particular this implies that the spectral multiplicity of $\wt T_0$ does not exceed $\dim \bH_0$.
\end{theorem}
\begin{corollary}\label{cor5.6}
Let under the assumptions {\rm (A1)-(A3)} $\tau $ be an admissible boundary parameter, let  $\s(\cd)=\s_\tau(\cd)$ be a pseudospectral function (with respect to $U$) and let $V_{0,\s}(=V_\s\up\gH_0)$ be  the corresponding isometry from $\gH_0$ to $\LS$ . Then $V_{0,\s}$ is a unitary operator if and only if the parameter  $\tau$ is  self-adjoint. If this condition is satisfied, then the boundary conditions  \eqref{4.17.1} defines an extension $\wt T^\tau\in\Selo$ and the operators $\wt T_{0,\tau}$ (the operator part of $\wt T^\tau$) and $\L_\s$ are unitarily equivalent by means of $V_{0,\s}$
\end{corollary}
\begin{proof}
The first statement is a consequence of Proposition \ref{pr3.11} and Theorem \ref{th4.7}. The second statement is implied by Theorems \ref{th4.7} and \ref{th5.5}.
\end{proof}
The criterion which enables one to describe all pseudospectral functions in terms of an arbitrary (not necessarily admissible) boundary parameter  is given in the  following theorem .
\begin{theorem}\label{th5.7}
The following statements  are equivalent:

{\rm (1)} each boundary parameter $\tau$ is admissible;

{\rm (2)} $\lim\limits_{y\to +\infty} \dot M_+(i y)\up\dot\cH_1 =0$  and
$\lim\limits_{y\to +\infty} y\left (\im (\dot M_+(iy)h,h)_{\dot\cH_0}+\tfrac 1 2 ||\dot P_2 h||^2\right)=+\infty$,

where $h\in\dcH_0, \; h\neq 0$ and $\dot P_2$ is the orthoprojection in $\dcH_0$ onto $\dcH_2:=\dcH_0\ominus \dcH_1$;

{\rm (3)} $\mul T=\mul T^*$, i.e., the condition (C2) in Assertion \ref{ass4.2.0} is fulfilled;

{\rm (4)} statement of Theorem \ref{th5.4} holds for arbitrary boundary parameters $\tau$.
\end{theorem}
\begin{proof}
Proposition \ref{pr5.2} and \eqref{2.8} yield the equivalence (1)$\Leftrightarrow$ (3). Since by Lemma \ref{lem4.11}, (2) $\dot M_+(\cd)$ is the Weyl function of the boundary triplet $\dot\Pi$, the equivalence (2)$\Leftrightarrow$ (3) is implied by \cite[Theorem 4.6]{Mog13.2}. The equivalence (1) $\Leftrightarrow$(4) follows from Theorem \ref{th5.4}.
\end{proof}
Combining the results of this section with Proposition \ref{pr3.10} we get the following theorem.
\begin{theorem}\label{th5.8}
Let the assumptions {\rm (A1)} and {\rm (A2)} be satisfied. Then the set of spectral functions of the system \eqref{3.3} (with respect to $U\in [\bH_0,\bH]$) is not empty if and only if $\mul T=\{0\}$  or equivalently if and only if the condition (C1) in Assertion \ref{ass4.2.0} is fulfilled.  If this condition is satisfied, then the sets of spectral and pseudospectral functions of the system \eqref{3.3} coincide and hence Theorems \ref{th5.4}, \ref{th5.5}, \ref{th5.7} and Corollary \ref{cor5.6} are valid for spectral functions (instead of pseudospectral ones). In this case
$\wt T_0,\; \wt T_{0,\tau}$ and $V_{0,\s}$ in Theorem \ref{th5.5} and Corollary \ref{cor5.6} should be replaced with $\wt T,\; \wt T_{\tau}$ and $V_{\s}$ respectively. Moreover, in this case statement {\rm (3)} in Theorem \ref{th5.7} takes the following form:

{\rm (3')} $\mul T^*=\{0\}$, i.e., the condition {\rm (C3)} in Assertion \ref{ass4.2.0} is fulfilled.
\end{theorem}
\begin{remark}\label{rem5.7}
Assume that $N_-\leq N_+$ and $\t$ is a subspace in $\bH$ such that $\t^\tm\in Sym (\bH)$ and system \eqref{3.3} is $\t$-definite. Moreover, let $\bH_0'$ be a subspace in $\bH$ and let $K_\t \in [\bH_0',\bH]$ be an operator with $\ker K_\t=\{0\}$ and $K_\t\bH_0'=\t$. It follows from Proposition \ref{pr3.10.1} and Remark \ref{rem3.10.2} that   Theorems \ref{th5.4}, \ref{th5.5},\ref{th5.7}, \ref{th5.8} and Corollary \ref{cor5.6} are valid, with some corrections, for pseudospectral and spectral functions $\s(\cd)$ with respect to $K_\t$ in place of $U$. We leave to the reader the precise formulation  of the specified results.
\end{remark}
\section{The case of the minimally possible $\dim\t$. Spectral functions of the minimal dimension.}
It follows from Lemma \ref{lem3.1.1}, (1) that the minimally possible dimension of the subspace $\t\subset \bH$ satisfying the assumption (A1) in Section \ref{sect4.1} is
\begin{gather}\label{6.1}
\dim\t=\nu+\wh\nu.
\end{gather}
If $\t$ satisfies (A1) and \eqref{6.1} then the previous results become essentially simpler. Namely, in this case the subspace $\bH_0$ from assumption (A2) satisfies $\dim\bH_0=\dim (H\oplus\wh H )$ and hence $H_1=\{0\}, \; H_1^\perp=H$ and
\begin{gather}\label{6.2}
\bH_0=H\oplus\wh H.
\end{gather}
Therefore the assumption (A2) in Section \ref{sect4.1}  takes the following form:

${\rm (A2')}$ $\bH_0$ is the subspace \eqref{6.2}, $\wt U$ and $\G_a$ are the same as in the assumption (A2) and
\begin{gather}\label{6.2.1}
\G_a=(\G_{0a}, \, \wh\G_a, \G_{1a})^\top:\dom\tma\to H\oplus\wh H\oplus H
\end{gather}
is the block representation of $\G_a$.

Below we suppose (unless otherwise is stated) the following assumption ${\rm(A_{min})}$, which is equivalent to the assumptions (A1) - (A3) and  the equality \eqref{6.1}:

${\rm(A_{min})}$ In addition to (A1) the  equality \eqref{6.1} holds and the assumptions ${\rm (A2')}$ and (A3) are satisfied.

Under this assumption the equalities \eqref{4.8} take the form
\begin{gather}\label{6.3}
\dot\cH_0 = \wh H\oplus\wt\cH_b, \qquad \dot\cH_1 = \wh H\oplus \cH_b
\end{gather}
and a boundary parameter is the same as in definition \ref{def4.6}.
\begin{theorem}\label{th6.1}
Let  $\tau$ be a boundary parameter \eqref{4.15} and let
\begin{gather*}
C_0(\l)=( \wh C_0(\l),  C_{0b}(\l)):\wh H\oplus\wt\cH_b\to\dcH_0,\quad
C_1(\l)=( \wh C_1(\l),  C_{1b}(\l)):\wh H\oplus\cH_b\to\dcH_0
\end{gather*}
be the block representations of $C_0(\l)$ and $C_1(\l)$.
Then for each $\l\in\bC_+$ there exists a unique pair of operator solutions $\xi_{\tau}(\cd,\l)\in\lo{H}$ and  $\wh\xi_{\tau}(\cd,\l)\in\lo{\wh H}$ of the system \eqref{3.3} satisfying the boundary conditions
\begin{gather}
\G_{1a}\xi_{\tau}(\l)=-I_H\label{6.4}\\
[(i {\wh C}_0(\l)- \tfrac 1 2 {\wh C}_1(\l))\wh \G_a + C_{0b}(\l)\G_{0b} -(i {\wh C}_0(\l)+ \tfrac 1 2 {\wh C}_1(\l))\wh \G_b + C_{1b}(\l)\G_{1b}]\xi_{\tau}(\l)=0 \label{6.5}\\
\G_{1a}\wh\xi_{\tau}(\l)=0\label{6.6}\\
[(i {\wh C}_0(\l)- \tfrac 1 2 {\wh C}_1(\l))\wh \G_a + C_{0b}(\l)\G_{0b} -\qquad\qquad\qquad\qquad\qquad\qquad\qquad\qquad\qquad\label{6.7}\\
\qquad\qquad\qquad\qquad\qquad -(i {\wh C}_0(\l)+ \tfrac 1 2 {\wh C}_1(\l))\wh \G_b +  C_{1b}(\l)\G_{1b}]\wh\xi_{\tau}(\l)={\wh C}_0(\l)+\tfrac i 2 {\wh C}_1(\l)\nonumber
\end{gather}
\end{theorem}
\begin{proof}
Let $v_\tau(\cd,\l)\in\lo{\bH_0}$ be the solution of \eqref{3.3} defined in theorem \ref{th4.12} and let
\begin{gather}\label{6.8}
v_\tau(t,\l)=(\xi_{\tau}(t,\l), \, \wh\xi_{\tau}(t,\l)):H\oplus\wh H\to \bH
\end{gather}
be the block representation of $v_\tau(t,\l)$. Then the first condition in \eqref{4.49} takes the form $\G_{1a}(\xi_\tau(\l), \wh\xi_\tau(\l))=(-I_H, 0)$, which is equivalent to \eqref{6.4} and \eqref{6.6}. Moreover, \eqref{4.9}, \eqref{4.10} and \eqref{4.48.4} take the form
\begin{gather*}
\dG_0'=( i(\wh\G_a-\wh\G_b), \, \G_{0b} )^\top,\quad
\dG_1'=( \tfrac 1 2 (\wh\G_a+\wh\G_b), \, -\G_{1b} )^\top,\quad \Phi(\l)=(0,\, \wh C_0(\l)+\tfrac i 2 \wh C_1(\l)).
\end{gather*}
Therefore the second condition in \eqref{4.49} is equivalent to \eqref{6.5} and \eqref{6.7}. Now the required statement is implied by Theorem \ref{th4.12}.
\end{proof}
It follows from \eqref{6.2}, \eqref{6.2.1} and \eqref{4.52} that $P_{\bH,\bH_0}\G_a=(\G_{0a}, \, \wh\G_a)^\top$ and $J_0=\begin{pmatrix} 0&0\cr 0 & iI_{\wh H} \end{pmatrix}$. This and \eqref{4.53} imply that in the case \eqref{6.1} (i.e., under the assumption ${\rm (A_{min})}$) the $m$-function $m_\tau(\cd)$ can be defined as
\begin{gather*}
m_{\tau}(\l)= \begin{pmatrix} \G_{0a}\xi_{\tau}(\l) & \G_{0a}\wh\xi_{\tau}(\l)\cr \wh\G_a \xi_{\tau}(\l)& \wh\G_a \wh\xi_{\tau}(\l)+\tfrac i 2 I_{\wh H}\end{pmatrix}:\underbrace{H\oplus\wh H}_{\bH_0}\to \underbrace{H\oplus\wh H}_{\bH_0},\quad \l\in\bC_+,
\end{gather*}
The following proposition is implied by Proposition \ref{pr4.9}.
\begin{proposition}\label{pr6.2}
For any $\l\in\bC_+$ there exists a unique collection of operator solutions $\xi_0(\cd,\l)\in\lo{H}, \; \wh\xi_0(\cd,\l)\in\lo{\wh H}$ and $u_+(\cd,\l)\in\lo{\wt\cH_b}$ of the system \eqref{3.3} satisfying the boundary conditions
\begin{gather*}
\G_{1a}\xi_0(\l)=-I_H, \qquad \wh\G_a \xi_0(\l)=\wh\G_b \xi_0(\l), \qquad \G_{0b}\xi_0(\l)=0\\
\G_{1a}\wh\xi_0(\l)=0, \qquad  i(\wh\G_a-\wh \G_b)\wh\xi_0(\l)=I_{\wh H}, \qquad \G_{0b}\wh\xi_0(\l)=0\\
\G_{1a}u_+(\l) =0,\qquad    \wh\G_a u_+(\l)=\wh\G_b u_+(\l),\qquad  \G_{0b}u_+(\l) =I_{\wt\cH_b}.
\end{gather*}
\end{proposition}
If the assumption ${\rm (A_{min})}$ is satisfied, then the operator function $M_+(\cd)$ from Proposition \ref{pr4.10} takes the form
\begin{gather*}
M_+(\l)=\begin{pmatrix} M_{11}(\l) & M_{12}(\l) & M_{13}(\l)\cr M_{21}(\l) & M_{22}(\l) & M_{23}(\l) \cr  M_{31}(\l) & M_{32}(\l) & M_{33}(\l)\end{pmatrix}: H\oplus \wh H \oplus\wt\cH_b\to H\oplus \wh H \oplus \cH_b,
\end{gather*}
where $\l\in\bC_+$  and
\begin{gather*}
M_{11}(\l)=\G_{0a}\xi_0(\l), \quad M_{12}(\l)= \G_{0a}\wh\xi_0(\l), \quad M_{13}(\l)=\G_{0a}u_+(\l)\\
M_{21}(\l)=\wh\G_{a}\xi_0(\l), \quad M_{22}(\l)=\wh \G_{a}\wh\xi_0(\l)+\tfrac i 2 I_{\wh H}, \quad M_{23}(\l)=\wh\G_{a} u_+(\l)\\
M_{31}(\l)=-\G_{1b}\xi_0(\l), \quad M_{32}(\l)=-\G_{1b}\wh\xi_0(\l),  \quad M_{33}(\l)=-\G_{1b}u_+(\l).
\end{gather*}
Moreover, the operator functions $m_0(\cd), \; S_1(\cd), \; S_2(\cd)$ and $\dot M_+(\cd)$ in Theorem \ref{th5.4} take the following simpler  form (cf. \eqref{4.40}-\eqref{4.43}):
\begin{gather*}
m_0(\l)=\begin{pmatrix}M_{11}(\l)&M_{12}(\l)\cr M_{21}(\l)&M_{22}(\l)  \end{pmatrix}:\underbrace{H\oplus\wh H}_{\bH_0}\to \underbrace{H\oplus\wh H}_{\bH_0}, \quad \l\in\bC_+ \\
S_1(\l)=\begin{pmatrix} M_{12}(\l)& M_{13}(\l)\cr M_{22}(\l)-\tfrac i 2 I_{\wh H} & M_{23}(\l)\end{pmatrix}:\underbrace{\wh H\oplus \wt\cH_b}_{\dot\cH_0}\to\underbrace{H\oplus\wh H}_{\bH_0},\quad \l\in\bC_+\\
S_2(\l)=\begin{pmatrix} M_{21}(\l)& M_{22}(\l)+\tfrac i 2 I_{\wh H}\cr M_{31}(\l)&M_{32}(\l) \end{pmatrix}:\underbrace{H\oplus\wh H}_{\bH_0}\to \underbrace{\wh H\oplus\cH_b}_{\dot\cH_1},\quad \l\in\bC_+ \\
\dot M_+(\l)=\begin{pmatrix} M_{22}(\l) &M_{23}(\l) \cr M_{32}(\l) &M_{33}(\l) \end{pmatrix}:\underbrace{\wh H\oplus\wt\cH_b}_{\dot\cH_0}\to \underbrace{\wh H\oplus\cH_b}_{\dot\cH_1},\quad \l\in\bC_+.
\end{gather*}
In the following theorem we characterize spectral functions of the  minimal dimension.
\begin{theorem}\label{th6.3}
Let system \eqref{3.3} be definite (see Definition \ref{def3.11.1}) and let  $N_-\leq N_+$. Then the following statements are equivalent:

{\rm (i)} $\mul \Tmi=\{0\}$, i.e., the condition {\rm (C0)} in Assertion \ref{ass4.2.0} is fulfilled;

{\rm (ii)} The set of spectral functions of the system is not empty, i.e., there exist subspaces $\t$ and $\bH_0'$ in $\bH$ and a spectral function $\s(\cd)$ of the system (with respect to $K_\t\in \in [\bH_0',\bH]$).

If the  statement {\rm (i)} holds, then the dimension $n_\s$ of each spectral function $\s(\cd)$ (see Definition \ref{def3.9}) satisfies
\begin{gather}\label{6.14}
\nu+\wh\nu\leq n_\s\leq n
\end{gather}
and there exists a spectral function $\s(\cd)$ with the minimally possible dimension $n_\s=\nu+\wh\nu$.
\end{theorem}
\begin{proof}
Assume  statement (i). Then by Lemma \ref{lem3.3} there exists a subspace $\t\subset \bH$ such that $\t^\tm\in Sym (\bH)$, $\dim\t=\nu+\wh\nu$ and the relation $T$ of the form \eqref{3.8} satisfies $\mul T=\{0\}$.   Therefore by Corollary \ref{cor5.3.1}, (2) and Proposition \ref{pr3.10} there exists a spectral function $\s(\cd)$ (with respect to $K_\t$). Moreover, $n_\s(=\dim\t)=\nu+\wh\nu$.

Next assume that $\t$ is a subspace in $\bH$ and $\s(\cd)$ is a spectral function (with respect to $K_\t$). Since the system is definite, it follows from Proposition \ref{pr3.12} that $\t^\tm\in Sym (\bH)$. Therefore by Lemma \ref{lem3.1.1}, (1) $n_\s(=\dim\t)\geq\nu+\wh\nu$, which yields \eqref{6.14}.

Conversely, let statement (ii) holds. If $\s(\cd)$ is a  spectral function (with respect to $K_\t$), then according to Proposition \ref{pr3.10} $\mul T=\{0\}$. This and the obvious inclusion $\mul\Tmi\subset \mul T$ yield statement (i).

\end{proof}

\end{document}